\documentclass{amsart}
\usepackage[utf8]{inputenc}
\usepackage{graphicx,graphics}
\usepackage[title]{appendix}
\usepackage{amsmath,amssymb}
\usepackage{hyperref}
\usepackage{epstopdf}
\usepackage{bbm}

\usepackage{xcolor}
\usepackage{enumerate}
\usepackage{mathtools}
\usepackage{color,xcolor} 
\usepackage{mathrsfs,bbm} 
\DeclareMathAlphabet\mathbfcal{OMS}{cmsy}{b}{n}
\DeclareMathOperator*{\esssup}{ess\,sup}
\DeclareMathOperator*{\essinf}{ess\,inf}

\renewcommand{\Re}{\text{Re }}
\renewcommand{\Im}{\text{Im }}

\newtheorem{lemma}{Lemma}
\newtheorem{definition}{Definition}
\newtheorem{proposition}{Proposition}
\newtheorem{remark}{Remark}
\newtheorem{corollary}{Corollary}
\newtheorem{theorem}{Theorem}

\numberwithin{equation}{section}
\numberwithin{figure}{section}
\numberwithin{theorem}{section}
\numberwithin{lemma}{section}
\numberwithin{proposition}{section}
\numberwithin{corollary}{section}
\numberwithin{remark}{section}
\numberwithin{definition}{section}

\newcommand{\JN} [1]{{\color{black} #1}}

\newenvironment{numberedproof}[2][Proof]{\noindent \emph{#1 #2} }{\hfill \qed}

\newcommand{\wu}{\widehat{u}}
\newcommand{\wv}{\widehat{v}}

\newcommand{\VK}{\widehat{V}_K}
\newcommand{\VhK}{\widehat{V}_{h,K}}
\newcommand{\wuK}{\widehat{u}^K}
\newcommand{\uK}{u^K}
\newcommand{\wvK}{\widehat{v}^K}
\newcommand{\vK}{v^K}
\title[Floquet--Bloch analysis for wave propagation]{Floquet--Bloch Analysis of Wave Propagation with Time-Periodic Coefficients}
\author{J\"org Nick, Ralf Hiptmair and Habib Ammari}

\address{
	Seminar for Applied Mathematics,
	ETH Zürich, Rämistrasse 101, CH-8092 Zürich, Switzerland 
}
\email{joerg.nick@sam.math.ethz.ch}
\email{ralf.hiptmair@sam.math.ethz.ch}
\email{habib.ammari@sam.math.ethz.ch}

\newcommand{\rhc}[1]{{\color{magenta}{\bf R.H.}: #1}}

\newcommand{\drop}[1]{}

\begin{document}
	
	\maketitle
	
	\begin{abstract}
		This paper presents a numerical investigation of acoustic wave propagation in an obstacle with periodically time-modulated material parameters. We focus on the numerical construction of Floquet--Bloch solutions, which are quasi-periodic kernel elements of the hyperbolic operator appearing on the left-hand side of the acoustic wave equation. Using the temporal Fourier expansion yields a system of coupled harmonics, which can be truncated. Rewriting this system then provides different (generally nonlinear) eigenvalue formulations for discretized Floquet--Bloch solutions. Deriving energy estimates and the necessary conditions for Riesz--Schauder theory show basic properties of the occurring Floquet exponents. 
		To derive fully discrete schemes, we employ a general Galerkin space discretization. Under assumptions on the relation of the temporal Fourier truncation and the Galerkin space discretization, we prove that the approximated Floquet exponents exhibit the same limitations as their continuous counterparts. Moreover, the approximated modes are shown to satisfy the defining properties of Floquet--Bloch solutions, with a defect that tends to zero as the number of harmonics approaches infinity. 	Numerical experiments demonstrate the effectiveness of the proposed approach and illustrate the theoretical findings.

		
		%
		
	\end{abstract}
	\section{Problem formulation}
	
	\subsection{Motivation} Time-modulated metamaterials \cite{CD20,G22,YGA22} have attracted a lot of attention in recent years. Such micro-structured materials exhibit unusual properties such as double-near zero effective material parameters \cite{AH21}, non-reciprocity \cite{C19,ACH22}, signal amplification \cite{G19}, and frequency conversion \cite{WG20}. These properties are due to the high contrast of their constituents and the occurrence of resonance phenomena at subwavelength scales, i.e., scales that are much smaller than the operating wavelength. The mathematical analysis of time-modulated metamaterials has been initiated in recent papers \cite{AH21, ACHR23,ACHR24,Thea1,Thea2,HSW23}. In the present manuscript, we consider a general class of time-modulated scatterers and do not restrict ourselves to the high-contrast regime. 
	We consider a single connected domain with appropriate boundary conditions, where the material parameters inside the domain are time-periodic and variable in space. Floquet-Bloch solutions, which are (temporally) quasi-periodic acoustic waves in the time-periodic medium, give insight into solutions of the corresponding initial-value problem. The study of these special solutions and their numerical approximation is the overall objective of this manuscript.
	
	\subsection*{The contribution of this manuscript} Coupling a range of harmonics is a natural and classical approach to treating dynamical systems with periodic coefficients \cite{S78,W90,WH90,WB19,YXB19}. With the advent of time-dependent metamaterials, partial differential equations with time-periodic coefficients have become an active field of research. Coupled harmonics are a natural tool for approaching these wave propagation problems and understanding resonance phenomena, although a rigorous mathematical understanding is lacking in the current literature (see \cite{G22,HD24}). In this work, we provide the first results for coupled harmonics in the context of the acoustic wave equation. Here, Floquet--Bloch solutions are only known to exist for the spatially discretized system. The present paper shows that the system of coupled harmonics, formulated for the spatially discrete acoustic wave equation, produces fully discrete quasi-periodic functions that fulfill the spatially discrete acoustic wave equation up to a residuum that tends to zero. Moreover, key properties of the continuous system, like folding and upper bounds on the energy growth, are shown to be present in the discrete system. With these results, we provide the first steps towards a complete numerical analysis of coupled harmonics for the acoustic wave equation with time-periodic coefficients.
	
	\subsection{Initial value problems with time-modulated coefficients}
	In order to describe the problem setting in the present manuscript, we let $D\subset \mathbb R^{d}$ be a bounded Lipschitz domain and consider the scalar wave equation with time-dependent material parameters
	\begin{align}\label{eq:time-varying-acoustic}
		\partial_t^2  u(x,t)+ A(x,t) u(x,t)  = f(x,t), \quad x\in D, t \ge 0 . 
	\end{align}
	The time-periodic and spatially elliptic operator reads
	$$A(x,t) u = -\nabla \cdot \left( \kappa(x,t)\nabla u \right), \quad\text{with} \quad \kappa(x,t)=\frac{c^2}{n^2(x,t)}, $$
	where $c$ denotes the wave speed in the background medium and $n(x,t)$ the refractive index. Throughout this manuscript, we assume that $\kappa(x,t)$ is regular in time, $T$-periodic for some $T>0$ and is positive (see assumption \eqref{assumpt-kappa}).
	
	
	The model problem is completed with one of the two following boundary conditions:
	Firstly, we consider absorbing boundary conditions, namely 
	\begin{align}\label{eq:Absorbing}
		\kappa 	\partial_\nu u = \kappa_0\partial_t u  \quad \text{on} \quad \Gamma=\partial D,
	\end{align}
	for some $\kappa_0>0$, where $\partial_\nu$ denotes the outward normal derivative.
	In the one-dimensional setting, this boundary condition is perfectly radiating, which is equivalent to enforcing the acoustic wave equation with homogeneous physical parameters in the exterior domain. Generally, this boundary condition introduces some damping into the system, which is a simplified model of the loss of energy that radiates into an exterior domain $D^+ = \mathbb R^d\setminus D$.
	
	Alternatively, we consider homogeneous Neumann boundary conditions, namely 
	\begin{align}\label{eq:Neumann}
		\partial_\nu u = 0  \quad \text{on} \quad \Gamma=\partial D.
	\end{align}
	In this model problem, the boundary condition does not introduce absorption; the only energy variation arises from the time modulation of $\kappa$.
	
	To complete the formulation of the problem of \eqref{eq:time-varying-acoustic}, we require an additional constraint in time. Two different temporal conditions are naturally of interest and are described below.
	
	\subsubsection*{Perspective I: Evolution problem}
	We complete the acoustic wave equation with time-dependent coefficients \eqref{eq:time-varying-acoustic}, by one of the boundary conditions \eqref{eq:Absorbing}/\eqref{eq:Neumann} to a well-posed evolution problem by imposing initial values, namely $$u(\cdot,0) = u_0 \quad\text{and}\quad \partial_t u(\cdot,0) = v_0.$$ 
	The resulting initial value problem is well-posed and stable for sufficiently regular initial values $(u_0,v_0)\in H^1(D)^2$ and can be analyzed with techniques based on semigroup theory, which has been conducted in \cite{K73}. Here, $H^1(D)$ denotes the usual Sobolev space of square-integrable functions whose weak \JN{derivatives are} square-integrable in $D$. 
	
	\subsubsection*{Perspective II: Quasi-periodic temporal conditions}
	Alternatively, and of particular interest in the design of Floquet metamaterials, is the existence, construction, and manipulation of \emph{Floquet--Bloch solutions}: We seek $\omega\in\mathbb C$ and an accompanying $u$ that fulfills \eqref{eq:time-varying-acoustic} in $D$ with $f=0$, one of the boundary conditions \eqref{eq:Absorbing}/\eqref{eq:Neumann} on the boundary $\Gamma$ and has the form
	\begin{align}
		u^{} = \Re e^{-i\omega t} u^{}_{\omega} ,
	\end{align}
	where $u_{\omega}$ is complex-valued and T-periodic. 
	Let $\Omega = 2\pi/T$ be the frequency of the time modulation. 
	We note that, by construction, these values $\omega$ are only defined up to translations of $n\Omega$ for $n\in \mathbb Z$. Consequently, we write that the quasi-frequency $\omega$ is an element of the temporal Brillouin zone $\mathbb T = \mathbb C /\Omega \mathbb Z$ and has a representation with real part in the interval $(-\Omega/2,\Omega/2]$.

	
	In this manuscript, we are mostly interested in the second \JN{perspective}, namely the numerical approximation of the Floquet exponents and their corresponding quasi-periodic solutions. However, the corresponding initial value problem and its analysis provide key insights into modes of the coupled harmonics and clarifies what Floquet exponents can be expected. 
    
	\subsection{Weak formulations}
    
	We start by deriving a weak formulation of the initial value problem with the absorbing boundary condition \eqref{eq:Absorbing}. Multiplying \eqref{eq:time-varying-acoustic} by a test function $v\in H^1(D)$ yields the identity
	\begin{align*}
		\left(v, \partial_t^2  u \right)_{L^2(D)} -\left(v, \nabla \cdot \kappa \nabla u\right)_{L^2(D)}  = \left (v,f\right)_{L^2(D)}. 
	\end{align*}
	We use Green's formula for the second summand and insert the absorbing boundary condition to arrive at the following weak formulation: We seek $u: [0,\widetilde T] \rightarrow H^1(D)$, such that for $0 \le t\le \widetilde T$ we have, for all $v \in H^1(D)$: 
	\begin{align}\label{eq:time-varying-acoustic-weak-absorbing}
		\left(v, \partial_t^2  u \right)_{L^2(D)} + \left(\nabla v,  \kappa \nabla u\right)_{L^2(D)} 
		+ \left( v, \kappa_0 \partial_t u \right)_{L^2(\Gamma)} = \left ( v,f\right)_{L^2(D)}. 
	\end{align}
	Throughout the paper, we \JN{make} at least the following assumptions on the modulation:
	\begin{align}\label{assumpt-kappa}
		\begin{aligned}
			\kappa &\in C^1_{\text{per}}(0,T;L^{\infty}(D))
			\\ 
			\quad 0 &< c_\kappa = \essinf_{x\in D, t\in[0,T]} \kappa(x,t)\le  \esssup_{x\in D, t\in[0,T]} \kappa(x,t)=C_\kappa<\infty.
		\end{aligned}
	\end{align}
	\begin{remark}
		We note that the weak formulation for the homogeneous Neumann boundary condition \eqref{eq:Neumann} reads: Find $u: [0,\widetilde T] \rightarrow H^1(D)$, such that for all $0 \le t\le \widetilde T$  we have, for all $v \in H^1(D)$: 
		\begin{align}\label{eq:time-varying-acoustic-weak-neumann}
			\left(v, \partial_t^2  u \right)_{L^2(D)} + \left(\nabla v,  \kappa \nabla u\right)_{L^2(D)}  = \left ( v,f\right)_{L^2(D)}. 
		\end{align}
	\end{remark}
	
	\section{Energy estimates for the initial value problem}
    
	Our investigations start with energy considerations for the initial value problem with the absorbing boundary condition. We address the case of homogeneous boundary conditions, which lack absorption, at the end of the section.
	The energy of the acoustic wave, at a given time $t$, is given by 
	\begin{align}\label{energy}
		\mathcal E(t) = 	\frac{1}{2} 
		\int_D  \bigg( | \partial_t	u (x,t) |^2
		+\kappa(x,t)\left|\nabla u (x,t)\right|^2 \bigg) \mathrm d x.
	\end{align}
	When convenient, we drop the omnipresent spatial argument $x$ in $u$, $\kappa$ and $f$.

    \subsection{Bounds on $\mathcal{E}(t)$}
   	
	We have the following characterization for the energy associated with the initial value problem \eqref{eq:time-varying-acoustic-weak-absorbing}, formulated on the bounded domain $D\subset \mathbb R^d$, with absorbing boundary conditions. 
	\begin{lemma}\label{lem:energy-identity-with-excit} Let $u\in C^{2}(0,T;H^1(D)) $ solve the evolution problem \eqref{eq:time-varying-acoustic-weak-absorbing} corresponding to the time-modulated system with absorbing boundary conditions \eqref{eq:Absorbing} and further let $\kappa$ fulfill the assumptions \eqref{assumpt-kappa}. 
		Then, we have, for all $t>0$, the energy identity
		\begin{align*}
			\mathcal E(t)- \mathcal E(0)
			&=
			\int_0^t \int_D \bigg( \partial_t u (s) f (s)+ \kappa'(s) \left|\nabla u (s)\right|^2  \bigg) \mathrm d x
			-\kappa_0\int_\Gamma
			\left| \partial_t u (s)\right|^2 \mathrm d x \,  \mathrm d s
			.
		\end{align*}
	\end{lemma}
	\begin{proof}
		Testing \eqref{eq:time-varying-acoustic-weak-absorbing} with $v =\partial_t u$ yields for all $t\ge 0$
		\begin{align*}
			\frac{1}{2}\int_D  \partial_t\left| \partial_t	u(t)\right|^2 
			+ \kappa (t)\partial_t | \nabla u(t) |^2 \,\mathrm d x
			+ \kappa_0\int_\Gamma | \partial_t u (t)|^2 \,\mathrm d x
			& =\int_D \partial_t u (t)f (t)\,\mathrm d x.
		\end{align*}
		After integrating both sides from $0$ to $t$ and using partial integration for the modulated term, we obtain
		\begin{align*}
			&\frac{1}{2}\int_D  |\partial_t	u(t)|^2 - |\partial_t	u_0|^2
			+\kappa(t)|\nabla u(t) |^2- \kappa(0)|\nabla u_0 |^2\mathrm d x 
			\\
			&
			+\int_0^t\int_D-\kappa'(s) \left| \nabla u(s) \right|^2 \mathrm d x
			+\kappa_0\int_\Gamma
			\left| \partial_t u(s)\right|^2 \mathrm d x  \,\mathrm d s
			=\int_0^t \int_D \partial_t u (s)f (s)\,\mathrm d x \, \mathrm d s.
		\end{align*}
		Rearranging the terms yields the stated result.
	\end{proof}
	
	The previous result shows that the (externally powered) time modulation can both introduce energy into the system and induce damping. Moreover, we observe that the change in energy induced by the modulation depends on the magnitude of spatial oscillations present in the system.
	Using the energy identity and a Gronwall argument, we obtain an energy result that bounds the change in energy introduced by the modulation over a specified time $t$.
	\begin{proposition}\label{prop:energy-estimate}
In the setting of Lemma~\ref{lem:energy-identity-with-excit}, we have the energy estimate
		\begin{align*}
			\mathcal E(t) \le e^{C_{\kappa}' ( t+T)} \left(\mathcal E(0)
			+ \dfrac{1}{4{C_{\kappa}'}}\int_0^t|| f(s) ||^2_{L^2(D)}\mathrm ds\right) ,
		\end{align*}
		with the constant 
		\begin{align*}
			C_{\kappa}' = \dfrac{2}{T}\int_0^T \left\| \dfrac{\kappa'(s)_+}{\kappa(s)} \right\|_{L^{\infty}(D)} \mathrm d s.
		\end{align*}	
	\end{proposition}
	\begin{proof}
		For all $t>0$, omitting the term implied through the absorption together with the Cauchy--Schwartz inequality and finally the generalized Young's inequality yields
		\begin{align*}
			\mathcal E(t)&\le \mathcal E(0)
			+ \int_0^t \left( \partial_t u (s) , f (s)\right)_{L^2(D)}  +
			\left\| \dfrac{\kappa'(s)_+}{\kappa(s)} \right\|_{L^{\infty}(D)}
			\left\|\sqrt{\kappa}(s)\nabla u(s) \right\|_{L^2(D)}^2\mathrm d s .
			\\ & \le 
			\mathcal E(0)
			+ \int_0^t  \dfrac{1}{4{C_{\kappa}'}}|| f(s) ||^2_{L^2}  +
			\max\bigg({C_{\kappa}'} ,  \left\| \dfrac{\kappa'(s)_+}{\kappa(s)} \right\|_{L^{\infty}(D)}\bigg)\mathcal E(s) \mathrm d s .
		\end{align*}
		The result is now obtained by applying the integral form of Gronwall's inequality \cite{B43}, which gives
		\begin{align*}
			\mathcal E(t) & \le \left(\mathcal E(0)
			+ \dfrac{1}{4{C_{\kappa}'}}\int_0^t|| f(s) ||^2_{L^2(D)}\mathrm ds\right)\exp\left(\int_0^t \max\left({C_{\kappa}'} , 2 \left\| \dfrac{\kappa'(s)_+}{\kappa(s)} \right\|_{L^{\infty}(D)}\right)\mathrm d s \right)
			\\ &
			\le \left(\mathcal E(0)
			+ \dfrac{1}{4{C_{\kappa}'}}\int_0^t|| f(s) ||^2_{L^2(D)}\mathrm ds\right) e^{{C_{\kappa}'} ( t+T)} .
		\end{align*}
		
	\end{proof}

	With these estimates, we have the right tools at our disposal to give some basic results on Floquet exponents, which are described in the following. The energy estimates for a given geometry $\JN{D}$ and physical parameters $\kappa(x,t),\kappa_0$ are not necessarily sharp. The existence of exponentially growing modes in particular, which \JN{plays} a key role in applications like signal amplification, \JN{cannot} be established by this approach. The numerical study of Floquet exponents provides an effective approach for tasks of this type and are introduced in the following section.
	\subsection{Characterizations of Floquet exponents $\omega$}
	We say $\omega$ is a Floquet \JN{exponent} of the time-modulated system if there exists a complex-valued \JN{$T$-periodic} $u^\omega \in \JN{H^1_{\text{per}}(0,T;H^1(D))}$, such that 
	\begin{equation}\label{eq:Bloch}
		\Re e^{-i\omega t} u^{\omega}
	\end{equation}
	is a nontrivial solution to the corresponding initial value problem with no excitation \JN{($f=0$)}. Since Floquet--Bloch solutions lie in the kernel of the left-hand side of \eqref{eq:time-varying-acoustic}, we also refer to them as \emph{resonant modes} and to associated Floquet exponents $\omega$ as \emph{resonant quasi-frequencies}.
	Here, $H^1_{\text{per}}(0,T;H^1(D))$ denotes the set of functions that are locally in $H^1(\mathbb{R})$ and periodic with period $T$. 
	
	We have the following result for the Floquet exponents of the system with absorbing boundary conditions.
	\begin{lemma}\label{lem:FL-multipliers-absorption} Let $\kappa$ fulfill the assumptions \eqref{assumpt-kappa}.
		All Floquet exponents $\omega$ of the system \eqref{eq:time-varying-acoustic-weak-absorbing}, which weakly enforces absorbing boundary conditions, have bounded imaginary part, i.e.,
		\begin{align*}
			\normalfont\Im \omega \le {C_{\kappa}'}.
		\end{align*}
		with the constant $	{C_{\kappa}'}$ from Proposition~\ref{prop:energy-estimate}.
	\end{lemma}
	\begin{proof}
		The statement is a direct consequence of the energy estimate described in Proposition~\ref{prop:energy-estimate}, which prohibits the existence of exponentially growing solutions with a higher rate than $C_{\kappa}'$.
	\end{proof}
	
	For homogeneous Neumann boundary conditions, the time-reversed system is structurally identical to the forward problem (with modulation $\kappa(x,-t)$). Repeating the previous argument therefore yields the following result.
	
	\begin{lemma}\label{lem:FL-multipliers-neumann}
		Let $\kappa$ fulfill the assumptions \eqref{assumpt-kappa}.
		All Floquet exponents $\omega$ of the system with homogeneous Neumann boundary conditions \eqref{eq:time-varying-acoustic-weak-neumann} are located in a strip around the real axis, namely
		\begin{align*}
			|	\normalfont\Im \omega |\le {C_{\kappa}'},
		\end{align*}
		with the constant $	{C_{\kappa}'}$ from Proposition~\ref{prop:energy-estimate}.
	\end{lemma}
	\begin{proof}
		The statement is a direct consequence of the energy estimate described in Proposition~\ref{prop:energy-estimate}, together with the same argument for the temporally reversed solution $u(x,T-t)$.
	\end{proof}
	\begin{remark}\label{rem:const-FL-multiplier}
		We note that $u(x,t)=\chi_D$ is always a Floquet solution with $\omega = 0\in \mathbb T$, both for the absorbing boundary condition \eqref{eq:Absorbing} and for the homogeneous boundary condition \eqref{eq:Neumann}.
	\end{remark}

    \subsection{Spatial Semi-discretization}
	In the following, let $X_h\subset H^1(D)$ for $h\ge 0$ denote an abstract sequence of finite-dimensional subspaces of $H^1(D)$, with a mesh width $h$ that tends to zero. We then obtain the following weak formulation on the Galerkin subspace $X_h$.
	Find $u_h: [0,\widetilde T] \rightarrow X_h$, such that for all $0 \le t\le \widetilde T$  and $v_h \in X_h$, we have: 
	\begin{gather}\label{eq:weak-absorption-h}
		\left(v_h, \partial_t^2  u_h \right)_{L^2(D)} + \left(\nabla v_h,  \kappa \nabla u_h\right)_{L^2(D)} 
		+ \kappa_0\left( v_h,   \partial_t u_h \right)_{L^2(\Gamma)} = \left ( v_h,f\right)_{L^2(D)}. 
	\end{gather}
	
	We note that the previous energy considerations also apply to the spatial discretization, which is described in the following. 
	The energy of the spatially discrete acoustic wave, at a given time $t$, is given by 
	\begin{align}\label{eq:energy-h}
		\mathcal E_h(t) = 	\frac{1}{2} 
		\int_D  \bigg( | \partial_t	u_h (t)|^2
		+\kappa(t) \left|\nabla u_h (t) \right|^2 \bigg) \, \mathrm d x.
	\end{align}
	
	\begin{proposition}\label{prop:energy-estimate-h}
		Consider the setting of Lemma~\ref{lem:energy-identity-with-excit} with either of the boundary conditions \eqref{eq:time-varying-acoustic-weak-absorbing} / \eqref{eq:time-varying-acoustic-weak-neumann}. Then, we have, for all $t\ge 0$, the energy estimate
		\begin{align*}
			\mathcal E_h(t) \le e^{{C_{\kappa}'} ( t+T)} \left(\mathcal E_h(0)
			+ \dfrac{1}{4C_{\kappa'}}\int_0^t|| f(s) ||^2_{L^2(D)}\mathrm ds\right) .
		\end{align*}
		with the same constant ${C_{\kappa}'}$ from Proposition~\ref{prop:energy-estimate}.
	\end{proposition}
		%
	We say $\omega_h$ is a Floquet exponent of the spatially discrete time-modulated system if there exists a complex-valued $u^\omega_h \in H^1_{\text{per}}(0,T;X_h)$, such that $$\Re e^{-i\omega_h t} u_h^{\omega}$$ is a nontrivial solution to the corresponding initial value problem without excitation.
	From the Floquet--Lyapunov theory for ordinary differential equations, we obtain that solutions of this form fully determine the solutions of the associated initial-value problem \eqref{eq:weak-absorption-h}.
	\begin{lemma}[Essentially \textrm{\cite[Corollary 1.5]{K16}}]\label{lem:Floquet}
\JN{		Consider the unique solution to the spatial semi-discretization \eqref{eq:weak-absorption-h} without excitation, that is, $f=0$. Let $\mathcal{F}_h$ denote the set of spatially discrete Floquet exponents, which depend on the Galerkin subspace $X_h$, or the mesh width $h$. For each $\omega_h$ there exists a set of periodic functions $x_{\omega_h,j}\in H^1_{\text{per}}(0,T; X_h)$ for $0\le j\le r$ with some $ r \in \mathbb N_0$, such that
		\begin{align*}
			u_h (t) = \sum_{\omega_h \in \mathcal{F}_h} 
			 e^{-i\omega_h t} \sum_{j \le r} x_{\omega_h,j}(t)t^j.
		\end{align*}
		Floquet exponents $\omega_h$ that are associated with some $r>1$ are exceptional points of the associated monodromy matrix and the coefficients $x_{\omega_h,j}$ are related to the corresponding generalized eigenvectors.}
	\end{lemma}
Such a result is not known for the acoustic wave equation itself.  Studying approximations of Floquet--Bloch solutions of the spatially discrete system is therefore a natural point of interest, since the long time behavior of spatially discrete solutions is determined by Floquet--Bloch solutions whose Floquet exponents have the largest imaginary part.

	\section{The system of coupled harmonics}
    
	Floquet--Bloch solutions of the system \eqref{eq:time-varying-acoustic} are of particular interest in the design of time-dependent metamaterials \cite{YGA22}. Here, we give some results on a classical approach to approximate these special solutions, which uses a harmonic expansion of the periodic function in \eqref{eq:Bloch}. Inserting the Floquet ansatz into the weak formulation with vanishing right-hand side reads: Find $\omega\in \mathbb C$ and a nontrivial $$u^\omega\in V_{\text{per}} =  H^1_{\text{per}}(0,T;L^2(D)) \cap L^2_{\text{per}}(0,T;H^1(D)),$$ such that for all $t\ge 0$ and $v\in H^1(D)$ we have 
	\begin{multline}
    \label{eq:weak-time-Floquet}
				\left(v, (-i\omega+\partial_t)^2  u^{\omega} \right)_{L^2(D)} + \left(\nabla v,  \kappa \nabla u^{\omega}\right)_{L^2(D)} +
			\\ \kappa_0\left( v, \kappa(x,t) (-i\omega+ \partial_t) u^{\omega} \right)_{L^2(\Gamma)} = 0
		\end{multline}
	This formulation is completed by time-periodic boundary conditions, which yields a quadratic eigenvalue problem on the bounded space-time domain $[0,T]\times D$.
	
	\begin{definition}
		We say that $\omega$ is a Floquet exponent of the system \eqref{eq:weak-time-Floquet}, if there exists a nontrivial $u^\omega\in V_{\text{per}}$ such that \eqref{eq:weak-time-Floquet} holds true. The associated solutions of the form \eqref{eq:Bloch} are referred to as Floquet--Bloch solutions.
	\end{definition}
	\noindent On $V_{\text{per}}$, we have the $L^2$- inner product 
	\begin{align}\label{eq:inner-product}
		\left(u, v\right)_{D} 	= \frac{1}{2\pi}\int_0^T \int_D  u(x,t)\overline{v}(x,t)\, \mathrm d x \, \mathrm d t = \sum_{n=-\infty}^\infty (\widehat u_n ,\widehat v_n)_{L^2(D)}.
	\end{align}
	Here, $\widehat u$ and $\widehat v$ denote the temporal Fourier coefficients of $u$ and $v$, respectively. The index $\infty$ corresponds to summation index on the right-hand side, for which we generally write 
    \JN{ 
    \begin{alignat*}{2}
       \left(\varphi,\psi\right)_{K,D} &= \sum_{n=-K}^K\left(\widehat{\varphi}_n,\widehat{\psi}_n\right)_{L^2(D)}
   , \ \
    \left(\varphi,\psi\right)_{D} &&= \sum_{n=-\infty}^\infty\left(\widehat{\varphi}_n,\widehat{\psi}_n\right)_{L^2(D)} ,
\\
    \left(\varphi,\psi\right)_{K,\Gamma} &= \sum_{n=-K}^K\left(\widehat{\varphi}_n,\widehat{\psi}_n\right)_{L^2(\Gamma)}
    , \ \
    \left(\varphi,\psi\right)_{\Gamma} &&= \sum_{n=-\infty}^\infty\left(\widehat{\varphi}_n,\widehat{\psi}_n\right)_{L^2(\Gamma)} .
    \end{alignat*}}
	We therefore arrive at the following formulation: 
    Find $\omega\in \mathbb C$ and $u^\omega\in V_{\text{per}}$, such that 
	\begin{align}\label{eq:weak-Floquet-time}
		\begin{aligned}
			&\left(v, (-i\omega+\partial_t)^2  u^\omega \right)_{D} + \left(\nabla v,  \kappa \nabla u^\omega\right)_{D} 
			\\&+ \kappa_0\left( v, (-i\omega+ \partial_t) u^\omega \right)_{\Gamma} 
			=0
		\end{aligned} 	
        \quad \text{for all}\quad v \in V_{\text{per}}
		. 
	\end{align}
\drop{	Here, we introduced a generally nontrivial right-hand side $b\in V_{\mathrm{per}}$, which may encode an excitation. Time-periodic formulations of this type determine the long-time behavior of the acoustic wave equations with constant coefficients (see, e.g., Appendix~\ref{sect:time-periodic}). For time-periodic modulations, such results are not known to the authors but still motivate the treatment of this slightly more general case. Constructing Floquet--Bloch solutions is then simply the study of singularities of the associated resolvent, as with standard time-harmonic techniques for wave equations with temporally constant coefficients.}
	The physical parameter $\kappa$ is $\Omega$-modulated and assumed to be fully given by a truncated Fourier expansion, i.e.,
	\begin{align*}
		A(x,t) =	\sum_{n=-K}^{K} \widehat{A}_n(x) e^{-i n\Omega t}, \quad \text{where} \quad \widehat A_n(x) \varphi = 
		-\nabla \cdot \left( \widehat{\kappa}_n(x)\nabla  \varphi \right).
	\end{align*}
	It is therefore natural and equivalent to formulate \eqref{eq:weak-time-Floquet} in terms of the Fourier coefficients, i.e., in the coefficients $\widehat{u}^\omega \in h^1(H^1(D))$, which determine $$u^\omega = \sum_{n=-\infty}^\infty \widehat u^\omega_n e^{-in\Omega t}.$$ Here, we used the sequence spaces $h^k(V)$, which are defined for $k\in \mathbb Z$ by
	\begin{align*}
		h^k(V) = \left\{ \left( \widehat u_n\right)_{n=-\infty}^\infty  | \sum_{n=-\infty}^\infty (1+|n|)^k \left\| \widehat u_n \right\|^2_V <\infty \right\}.
	\end{align*}  On these spaces, we define the operators $$\mathcal D \, \colon \, h^{k+1}(V) \rightarrow  h^{k}(V)\quad\text{and}\quad \mathcal{T}_\kappa \,\colon \,  h^{k}(V)\rightarrow h^{k}(V),$$
	which correspond to the temporal derivative $\partial_t$ and the modulation $\kappa(x,t)$ in the frequency domain and are determined through the expressions
	\begin{align*}
		\mathcal D \widehat{u} = \left( - in\Omega \widehat{u}_n\right)_{n\in \mathbb Z},
		\quad \mathcal T_\kappa \widehat{u} = \left(\sum_{m=-\infty}^\infty\widehat{\kappa}_{n-m}(x) \widehat u_{m}(x)\right)_{n\in \mathbb Z}.
	\end{align*}
	With this notation, we are in the position to formulate \JN{\eqref{eq:weak-Floquet-time} in the frequency domain}, for which we use the auxiliary space 
	$$\widehat{V}_{\mathrm{per}} = h^1(L^2(D))\cap h^0(H^1(D)).
	$$ Find $\omega\in \mathbb C$ and nontrivial $\widehat{u}^{\omega} \in \widehat{V}_{\mathrm{per}}\setminus \{0\}$, such that 
	\begin{align}\label{eq:weak-Floquet-freq}
		a_\omega(\widehat{u}^{\omega}, \wv) = 0, \quad \forall \wv\in \widehat{V}_{\mathrm{per}},
	\end{align}
	where the bilinear form $a_\omega \, \colon \widehat{V}_{\mathrm{per}}\times \widehat{V}_{\mathrm{per}} \rightarrow \mathbb C$ is given by
	\begin{align}\label{eq:weak-Floquet-forms}
		\begin{aligned}
			a_\omega(\wu, \wv) & = \left(\widehat{v}, (-i\omega+\mathcal D)^2  \wu \right)_{D} + \left(\nabla \widehat{v},  \mathcal T_\kappa \nabla \wu\right)_{D} 
			+ \kappa_0\left(\widehat{v},(-i\omega+ \mathcal D) \wu \right)_{\Gamma}.
		\end{aligned} 	
	\end{align}
	
	We note that here, we extend the gradient $\nabla$ naturally to the sequence space $\widehat{V}_{\mathrm{per}}$ by identifying it with the pointwise application to each sequence element.
	\begin{remark}
		We note that the corresponding formulation of the Neumann boundary condition is \eqref{eq:weak-Floquet-freq}, with the exception of the boundary form, which vanishes for homogeneous Neumann boundary conditions.
	\end{remark}
	
	\subsection{Frequency-truncated weak formulation} \label{sect:trunc}
	We discretize the weak formulation, by restricting the temporal weak formulation \eqref{eq:weak-Floquet-time} to the Galerkin subspace $$V_{K}= \text{span}\left\{ \phi(x)e^{-in\Omega t} \,\, |\,\,\phi \in H^1(D) , \, n \in \mathbb Z \, , \, |n|\le K\right\}.$$

	For any $K$, we denote by $(\cdot,\cdot)_{K,D}$ the canonical extension of the $L^2(D)$ pairing onto the sequence space, via \eqref{eq:inner-product}. For finite $K$, we naturally have $\VK\cong H^1(D)^{2K+1}$. For truncated functions
	\begin{align*}
		\uK(x,t) =
		\sum_{n=-K}^{K} \wuK_n(x) e^{-i n\Omega t},\quad
		\vK(x,t) =
		\sum_{n=-K}^{K} \wvK_n(x) e^{-i n\Omega t}  ,
	\end{align*}
    \JN{where we collect these Fourier coefficients in the vector
	$$
	\wuK  = \left(\widehat u_n \right)_{|n|\le K} .
	$$
	Parseval's formula gives that for the Fourier coefficients $\wuK,\wvK \in \VK \cong H^1(D)^{2K+1}$ we have}
	\begin{align}\label{eq:parseval}
		\left(\wuK,\wvK\right)_{K,D} =\left(\wuK,\wvK\right)_{D }	= \frac{1}{2\pi}\int_0^T \int_D \uK(x,t)\overline{\vK(x,t)}\, \mathrm d x \, \mathrm d t.
	\end{align}

	We obtain the following truncated problem formulation of (\ref{eq:weak-Floquet-freq}) in the Fourier domain: Find $\omega_K\in \mathbb C$ and $\widehat{u}_K^{\omega} \in \widehat{V}_K$, such that for all $\widehat v\in \widehat{V}_K$, we have
	\begin{align}\label{eq:weak-Floquet-freq-K}
		\begin{aligned}
			&\left(\widehat{v}, (-i\omega_K+\mathcal D)^2  \widehat{u}_K^{\omega} \right)_{K,D} + \left(\nabla \widehat{v},  \mathcal T_\kappa \nabla \widehat{u}_K^{\omega}\right)_{K,D} 
			\\&+ \kappa_0\left(\widehat{v}, (-i\omega_K+ \mathcal D) \widehat{u}_K^{\omega} \right)_{K,\Gamma} 
			= 0
			.
		\end{aligned} 	
	\end{align}
    \JN{We note that the entries of $\widehat{u}_K^{\omega}$ are denoted by $(\widehat{u}_K^{\omega})_n$ for $|n|\le K$. }
	%
	\begin{remark}
		This coupled system is a typical formulation in the engineering literature and is sometimes referred to as "harmonic balancing" (see, e.g. \cite{RMA03,WH90,YGA22}).
	\end{remark}
    
	\subsection{Weak formulation} 	
	\JN{We collect the left-hand side of \eqref{eq:weak-Floquet-freq-K} in the bilinear form:}
	\begin{align*}
		a^K_\omega(\wu, \wv ) &= \left(\widehat{v}, (-i\omega+\mathcal D)^2  \widehat{u} \right)_{K,D} + \left(\nabla \widehat{v},  \mathcal T_\kappa \nabla \widehat{u}\right)_{K,D} 
		+ \kappa_0\left(\widehat{v}, (-i\omega+ \mathcal D) \widehat{u} \right)_{K,\Gamma} .
		\end{align*}
	With \JN{this form}, we can write the weak formulation as: Find $\omega_K\in \mathbb C$ and $\wuK \in \widehat{V}_K$, such that
	\begin{align}\label{eq:weak-formulation-K}
		a^K_{\omega_K}(\wuK, \wv) = 0, \quad \forall \wv\in \VK.
	\end{align}
	We note that, by construction of the truncated system, we have the following bound.
	
	\begin{remark}
		\JN{Consider a regular function $ u^\omega\in C_{\text{per}}^p(0,T;H^1(D))$ whose Fourier coefficients $\widehat{u}^\omega$ fulfill  \eqref{eq:weak-Floquet-freq}, with $\omega\in \mathbb C$.  For $\kappa\in C_{\text{per}}^p(0,T;C^{\infty}(D))$, we then obtain, for any $\widehat v \in \VK$ that is normalized in the sense that $\|\widehat v \|^2_{K,D}+\|\nabla\widehat v \|^2_{K,D} = 1$
		\begin{align*}
			 a_{\omega}^K(\widehat{u}^\omega, \widehat{v}) = 
    	 a_{\omega}^K(\widehat{u}^\omega, \widehat{v})-	 a_{\omega}(\widehat{u}^\omega, \widehat{v})
            =-
           \sum_{|n|>K}\sum_{m=-K}^K \left( \nabla \widehat{v}_m, \widehat{\kappa}_{n-m} \nabla \widehat{u}^{\omega}_{m} \right)_{L^2(D)}
		\end{align*}
        Using the triangle inequality and the decay in the Fourier coefficients of $u^\omega$ and $\kappa$ that is implied by their temporal regularity and periodicity gives
       		\begin{align*}
			 |a_{\omega}^K(\widehat{u}^\omega, \widehat{v}) | 
          &\le
           C\sum_{|n|>K}\sum_{m=-K}^K \|\nabla \widehat{v}_m \|_{L^2(D)} (n-m)^{-p}(1+|m|)^{-p}   \| u^{\omega}_{m} \|_{C^p(0,T;H^1(D))}
            \\ &
            \le C K^{-p}\| u^{\omega}_{m} \|_{C^p(0,T;H^1(D))}.
		\end{align*}
		Here, $C_{\text{per}}^p(0,T;H^1(D))$ denotes the set of functions of class $C^p$ that are periodic in $t$ with values in $H^1(D)$.} 
	\end{remark}
	
	In the following, we show some basic properties of the sesquilinear form $a^K_\omega$. 
	\begin{lemma}[Continuity of $a^K_\omega$]\label{lem:bound-a}
		\JN{The bilinear form $a^K_\omega\, \colon \, \VK\times \VK \rightarrow \mathbb C$, fulfills the bound 
		\begin{align*}
			| a^K_\omega(\wu, \wv)| \le C (K+|\omega|)^2 \left(\|\wu \|_{K,D} \| \wv \|_{K,D} +  \|\nabla \wu \|_{K,D} \| \nabla \wv \|_{K,D} \right) , 
		\end{align*}
		for all $\wu,\wv\in \VK$. The constant $C$ depends on $\kappa$ and the boundary through the trace theorem.}
	\end{lemma}
	\begin{proof}
		We start with the triangle inequality and the Cauchy--Schwarz inequality to estimate the first summand
		\begin{align*}
			|a^K_\omega(\wu, \wv )| &\le (K+|\omega|)^2 \|\wu \|_{K,D} \| \wv \|_{K,D} + |\left(\nabla\widehat v, \mathcal T_\kappa \nabla\widehat u \right)_{K,D}|
			\\&+\kappa_0| \left(\widehat{v}, (-i\omega+ \mathcal D) \widehat{u} \right)_{K,\Gamma} |
			.
		\end{align*}
		For the second term, we use the Green's formula and Parseval's formula, to obtain
		\begin{align*}
			| \left(\nabla\widehat v, \mathcal{T}_\kappa \nabla\widehat u \right)_{K,D} | 
			&
			= \frac{1}{2\pi}
			\left| \int_{0}^T
			\int_D \kappa (x,t) \nabla \overline{v}(x,t) \cdot \nabla u(x,t)
			\, \mathrm d t
			\right| 
			\\ & \le 
			\frac{C_\kappa}{2\pi}
			\int_{0}^T
			\| \nabla v(t) \|_{L^2(D)} \| \nabla u(t) \|_{L^2(D)}
			\, \mathrm d t
			\\ & \le 
			\frac{C_\kappa}{2\pi} \int_{0}^T \frac{\rho}{2}
			\| \nabla v(t) \|^2_{L^2(D)} + \frac{1}{2\rho}\| \nabla u(t) \|^2_{L^2(D)}
			\, \mathrm d t . 
		\end{align*}
		The final estimate is obtained by using the generalized Young's inequality, which holds for all positive $\rho>0$. To conclude the estimation of the second term, we apply Parseval's inequality on the right-hand side again, which yields  
		\begin{align*}
			| \left(\nabla\widehat v, \mathcal{T}_\kappa \nabla\widehat u \right)_{K,D}  | 
			\le
			C_\kappa \left(
			\frac{\rho}{2}
			\| \nabla \wv \|^2_{K,D}  + \frac{1}{2\rho}\| \nabla \wu \|^2_{K,D}
			\right).
		\end{align*}
		The estimate for the second summand is obtained by $\rho = \| \nabla \wu \|_{K,D}\| \nabla \wv \|^{-1}_{K,D}$. (Note that the case $\| \nabla \wv \|_{K,D} = 0$ is trivial).
		
\JN{		For the final summand, which derives from the boundary terms, we obtain with the trace theorem
		\begin{align*}
			| \left(\widehat{v}, (-i\omega+ \mathcal D) \widehat{u} \right)_{K,\Gamma} |
			&\le (|\omega|+K)\ C_\Gamma^2 \left(\left\| \widehat{v}\right\|_{K,D}+\left\| \nabla\widehat{v}\right\|_{K,D}\right)\left( \left\| \widehat{u}\right\|_{K,D}+ \left\| \nabla \widehat{u}\right\|_{K,D}\right)
			.
		\end{align*}}

	\end{proof}
	Moreover, we have the following lower bound.
	\begin{lemma}[G\aa rding inequality]\label{lem:garding-a}
		We have the estimate
		\JN{\begin{align*}
			\mathrm{Re} \, \, a^K_\omega(\wu, \wu) \ge 
			\dfrac{c_\kappa}{2} \|\nabla \wu \|^2_{K,D} - C_\Gamma(K+|\omega|)^2 \|\wu \|^2_{K,D}+\kappa_0 \,\normalfont{\Im} \omega \,  \|\widehat{u} \|^2_{K,\Gamma},
		\end{align*}}
		\noindent with the constant $	c_\kappa$ from \eqref{assumpt-kappa}, for all $\widehat u \in \VK$. The constant $C_\Gamma$ depends on the boundary by the trace theorem and is defined in Lemma~\ref{lem:bound-a}.
	\end{lemma}
	\begin{proof}
		The proof uses similar arguments as those in Lemma~\ref{lem:bound-a}. We start again by the Cauchy--Schwarz inequality to obtain 
		\begin{align*}
			\mathrm{Re} \, \,\, a^K_\omega(\wu, \wu) \ge	
			\Re\left(\nabla \wu, \mathcal T_\kappa \nabla \widehat u \right)_{K,D}
			-(K+|\omega|)^2 \|\wu \|^2_{K,D} +\JN{ \Re \left(\widehat{u},  (-i\omega+ \mathcal D) \widehat{u} \right)_{K,\Gamma}  }  .
		\end{align*}
		The first summand is estimated by using the Parseval's formula, which yields
		\begin{align*}
			\Re\left(\nabla \wu, \mathcal T_\kappa \nabla \widehat u \right)_{K,D}
			&= \frac{1}{2\pi} \int_{0}^T
			\int_D \kappa(x,t) \nabla\overline{u}(x,t) \cdot\nabla u(x,t) \,\mathrm dx \, \mathrm dt
			\\ &\ge
			\frac{c_{\kappa}}{2\pi}  \int_{0}^T
			\int_D |\nabla u(x,t)|^2
			\,\mathrm dx \, \mathrm dt.
		\end{align*}
		\JN{A final application of Parseval's formula yields the stated result, up to the boundary terms. Finally, we have
        \begin{align*}
         \Re\left(\widehat{u},  (-i\omega+ \mathcal D) \widehat{u} \right)_{K,\Gamma} = \Im \omega \|\widehat{u} \|^2_{K,\Gamma},
        \end{align*}
        which gives the stated result.
        }
	\end{proof}
    \begin{remark}
   \JN{ For $\normalfont{\Im} \omega \ge 0$, the boundary term in the G\aa rding inequality can be neglected. The case $\normalfont{\Im} \omega < 0 $ corresponds for $K=0$ to a standard Helmholtz problem with non-absorbing boundary conditions. The boundary term in Lemma~\ref{lem:garding-a} is then still compact, since by the trace theorem \cite[Theorem 2.6.8]{SS11}, we have 
   \begin{align*}
       \| \widehat{u} \|^2_{K,\Gamma}
       \le 
       \sum_{n=-K}^K \| \widehat{u}_n \|^2_{H^{s}(\Omega)} \quad \normalfont{\text{for }}s>1/2.
   \end{align*}}
    \end{remark}

	These results provide the necessary conditions for applying Fredholm alternative arguments to an associated time-harmonic scattering problem.
    
	\begin{proposition}\label{prop:resonance}
		Either, the weak formulation
        \begin{align}\label{eq:coupled-system-steady-state}
            	a^K_\omega(\wu, \wv ) &= l(\wv), \quad  \forall \wv \in \VK
        \end{align}
        is uniquely solvable for all bounded linear forms $l \colon \VK\rightarrow \mathbb C$, or there exists a finite-dimensional subspace $E\subset \VK$, such that for $\wu \in E$ we have
		\begin{align*}  
			a^K_\omega(\wu,\wv) = 0,  \quad \forall \wv \in \VK.
		\end{align*}
		If $E$ is nontrivial, then the system \eqref{eq:weak-formulation-K} has at least one solution if and only if $\widehat b \perp E$.
	\end{proposition}
	\begin{proof}
    \	The results of Lemmas~\ref{lem:bound-a}--\ref{lem:garding-a} give the conditions of Fredholm--Riesz--Schauder theory (see, e.g., \cite[Theorem 2.1.60]{SS11}), which implies the stated result. 
	\end{proof}
    \begin{remark}
        Time-harmonic systems of the form \eqref{eq:coupled-system-steady-state} have been used to compute the frequency response of dynamical systems with periodic coefficients \cite{WH90} and are actively used to investigate the long-time behavior of wave scattering from time-dependent metamaterials \cite{HD24}.
    \end{remark}
    
	We obtain the following result on the set of (truncated) resonant quasi-frequencies.
	\begin{theorem}\label{thm:discrete}
		The set of truncated Floquet exponents $\omega_K$ is a discrete subset of $\mathbb C$, which does not contain an accumulation point. 
	\end{theorem}
	\begin{proof}
		By the continuity and the G\aa rding inequality of Lemmas~\ref{lem:bound-a}--\ref{lem:garding-a}, we are in the setting of the meromorphic Fredholm theorem \cite[Chapter 1, Theorem~1.16]{AZ18}. Therefore, it suffices to show that the bilinear form $a_\omega^K$ is injective at a point $\omega\in \mathbb C$. Setting $\omega = i(K+1)$ and using the Lax-Milgram Theorem then shows that the spectrum is discrete. 
	\end{proof}

	\begin{remark}
		We note that the resonant quasi-frequencies of the time-continuous system \eqref{eq:weak-Floquet-forms} are, due to folding, fully determined by those located in the first Brillouin zone. For a finite $K$, such a folding property generally does not hold, unless some filtering, e.g. through a space discretization, is applied (see Theorem~\ref{thm:discrete-folding}).  As $K\rightarrow \infty$, the \JN{number} of truncated resonant quasi-frequencies generally tends to infinity, causing a possibly infinite number of accumulation points to appear in the set of resonant quasi-frequencies (see Appendix~\ref{sect:1d-folding} for an example).
	\end{remark}

	\subsection{Formulations as eigenvalue problems}
	We note that \eqref{eq:weak-Floquet-freq-K} takes the form of a quadratic eigenvalue problem:
	Find $\omega_K \in \mathbb T$ and $\widehat{u}_K^{\omega} \in \VK$, such that for all $\widehat v\in \VK$, we have
	\begin{align}\label{eq:weak-Floquet-freq-quadratic}
		\begin{aligned}
			&-\omega_K^2\left(\widehat{v},   \widehat{u}_K^{\omega} \right)_{K,D}
			-i\omega_K \big(2\left(\widehat{v}, \mathcal D \widehat{u}_K^{\omega} \right)_{K,D}
			+\kappa_0 (\widehat{v}, \mathcal{T}_\kappa   \widehat{u}_K^{\omega} )_{K,\Gamma} 	\big)		
			\\&+\left(\widehat{v}, \mathcal D^2  \widehat{u}_K^{\omega} \right)_{K,D}
			+ \left(\nabla \widehat{v},  \mathcal T_\kappa \nabla \widehat{u}_K^{\omega}\right)_{K,D} 
			+\kappa_0\left(\widehat{v}, \mathcal{T}_\kappa  \mathcal D \widehat{u}_K^{\omega} \right)_{K,\Gamma} 
			= 0
			.
		\end{aligned} 	
	\end{align}
	This nonlinear eigenvalue problem can be linearized in different ways, depending on the setting. In the general setting pursued here, we can introduce the auxiliary variable $\widehat z^\omega_K = (-i\omega_K+\mathcal D) \widehat u^\omega_K$ (which corresponds to the temporal derivative of $u^\omega$). We then obtain the following linear formulation: Find $\omega_K\in \mathbb T$ and $(\widehat{u}_K^\omega, \widehat{z}_K^\omega)\in \VK^2,$ such that for all $(\widehat v, \widehat\eta) \in \VK^2$ it holds that
	\begin{align}\label{eq:block-linear-EV-1}
		\left(\widehat{v},  \mathcal D \widehat{u}_K^{\omega} \right)_{K,D} -		\left(\widehat{v}, \widehat{z}_K^{\omega} \right)_{K,D}&= 
		i\omega_K \left(\widehat{v},  \widehat{u}_K^{\omega} \right)_{K,D} ,
		\\
		\left(\widehat\eta,  \mathcal D  \widehat{z}_K^{\omega} \right)_{K,D} + \left(\nabla \widehat\eta,  \mathcal T_\kappa \nabla \widehat{u}_K^{\omega}\right)_{K,D} 
		+\kappa_0 (\widehat\eta, \mathcal{T}_\kappa   \widehat{z}_K^{\omega} )_{K,\Gamma}
		&= 	i\omega_K\left(\widehat\eta,   \widehat{z}_K^{\omega} \right)_{K,D} \label{eq:block-linear-EV-2}
		.	
	\end{align}
	Rewriting this weak formulation in block form reads
	\begin{align}\label{eq:block-EV}
		\left(\begin{pmatrix}
			\widehat{v} \\ \widehat{\eta}
		\end{pmatrix} ,  \begin{pmatrix}
			\mathcal D & - I \\ 
			\mathcal T_A
			& \mathcal D
		\end{pmatrix}
		\begin{pmatrix}
			\widehat{u}_K^{\omega} \\ \widehat{z}_K^{\omega}
		\end{pmatrix}  \right)_{K,D} 
		+\kappa_0 (\widehat\eta,    \widehat{z}_K^{\omega} )_{K,\Gamma}
		&= 
		i\omega_K
		\left(\begin{pmatrix}
			\widehat{v} \\ \widehat{\eta}
		\end{pmatrix}
		,
		\begin{pmatrix}
			\widehat{u}_K^{\omega} \\ \widehat{z}_K^{\omega}
		\end{pmatrix}
		\right)_{K,D}.	
	\end{align}
	Here, we use the notation  $$ \left( \widehat\eta,  \mathcal T_A \widehat{u}_K^{\omega}\right)_{K,D}  =  \left(\nabla \widehat\eta,  \mathcal T_\kappa \nabla \widehat{u}_K^{\omega}\right)_{K,D} . $$
		We note that the structural properties of \eqref{eq:weak-Floquet-freq-quadratic}, in particular G\aa rding's inequality of Lemma~\ref{lem:garding-a}, are not preserved by this formulation.

		\section{Fully discrete coupled harmonics eigenvalue problem}
            
		In the following, we consider the derivation of the truncated eigenvalue problem \eqref{eq:block-EV} in Section~\ref{sect:trunc}, but with the spatially discrete (and therefore finite-dimensional) Galerkin subspace
		$$\VhK= \text{span}\left\{ \phi_h(x)e^{-in\Omega t} \,\, |\,\,\phi_h \in X_h , \, n \in \mathbb Z \, , \, |n|\le K\right\}.$$
		Here, we assume that $X_h$ is a standard finite-dimensional subspace of $H^1(D)$, typically consisting of (piecewise) polynomial functions of a finite maximal polynomial degree. The parameter $h$ denotes the mesh width of a uniform grid, which tends to zero. 
        Throughout the rest of the paper, we denote by $C_{\text{inv}}(h)$ a constant given by an inverse estimate, such that we have 
		\begin{align}\label{eq:inverse-estimate}
			\|\nabla u_h\|_{L^2(D)} \le C_{\text{inv}}(h)	\|u_h\|_{L^2(D)}. 
		\end{align}
		This constant tends to infinity for $h\rightarrow 0$ and is crucially dependent on the approximation quality of the chosen subspace $X_h$. For finite element subspaces on polyhedral domains, we have $C(h)\propto h^{-1}$, which is found, e.g., in \cite[Theorem 4.5.11]{BS08}.
        
        We then find the following fully discrete eigenvalue formulation of \eqref{eq:block-EV}:	Find $\omega_{h,K} \in \mathbb T$ and $(\widehat{u}_{h,K}^{\omega},\widehat{z}_{h,K}^{\omega}) \in \VhK^2\subset \VK^2$, such that for all $(\widehat{v}_h,\widehat{\eta}_h)\in \VhK^2$ we have
		\begin{multline}
        \label{eq:block-EV-h}
							\left(\begin{pmatrix}
					\widehat{v}_h \\ \widehat{\eta}_h
				\end{pmatrix} ,  \begin{pmatrix}
					\mathcal D & - I \\ 
					\mathcal T_A
					& \mathcal D
				\end{pmatrix}
				\begin{pmatrix}
					\widehat{u}_{h,K}^{\omega} \\ \widehat{z}_{h,K}^{\omega}
				\end{pmatrix}  \right)_{K,D} 
				+\kappa_0 (\widehat\eta_h,  \widehat{z}_{h,K}^{\omega} )_{K,\Gamma}
				\\ = 
				i\omega_{h,K} 
				\left(\begin{pmatrix}
					\widehat{v}_h \\ \widehat{\eta}_h
				\end{pmatrix}
				,
				\begin{pmatrix}
					\widehat{u}_{h,K}^{\omega} \\ \widehat{z}_{h,K}^{\omega}
				\end{pmatrix}
				\right)_{K,D}.
		\end{multline}

		\subsection*{Decay of the temporal Fourier coefficients of spatially discrete modes}
		A key concern of the truncation used in \eqref{eq:weak-Floquet-freq-K} is the magnitude of the defect introduced by this procedure. In particular, we might be interested if a mode constructed by this procedure satisfies the homogeneous acoustic wave equation \eqref{eq:time-varying-acoustic} up to a small defect. For such results, we require that the truncation error is small, which only holds when $\widehat \kappa_n$ decays sufficiently fast, as $|n|\rightarrow \infty$ and, more critically, that the Fourier coefficients $ \widehat{u}_{h,K}^{\omega}$ fulfill a decay property with respect to $|n|$, as $|n|\rightarrow K$ and $K\rightarrow \infty$ (in the sense of \eqref{eq:local} below). From numerical evidence, we conclude that this is not the case for arbitrary fine space discretizations (see Figure~\ref{Fig:localization}). However, when sufficiently many temporal oscillations are resolved (i.e. $K$ is large enough for a fixed $h$), we observe that the Fourier coefficients of the eigenmodes of the system \eqref{eq:block-EV-h} decay. In this section, we give a mathematical explanation of this phenomenon.

		\begin{remark}
			The inverse estimate \eqref{eq:inverse-estimate} quantifies the filtering of spatial oscillations that is introduced by the space discretization. Alternatively, such a filtering condition could be imposed on the continuous truncated formulation \eqref{eq:weak-Floquet-freq-K}. Then the following results would \JN{carry} over to such an infinite-dimensional filtered setting. Here, they are presented in a unified setting for a generic discretization in space.
		\end{remark}
       
		We start with a result that gives worst-case estimates on the decay of the spectrum of fully discrete eigenmodes.
		
		\begin{theorem}\label{thm:localization}
			Let $\kappa$ fulfill the assumptions \eqref{assumpt-kappa} and, further, let $\widehat{u}_{h,K}^{\omega} $ and $\omega_{h,K} $ be a resonant mode and its associated quasi-frequency, which fulfill the eigenvalue problem \eqref{eq:block-EV-h}. Let $\omega_{h,K} $ be in the first Brillouin zone, i.e., have a real part in $(-\Omega/2,\Omega/2]$. Moreover, we assume that the mode is normalized with respect to $\left\| \cdot\right\|_{K,D}.$ For $1<|n|\le K$, we then obtain
			\begin{align}\label{eq:local}
				\left\| \left( \widehat{u}_{h,K}^{\omega}\right)_n \right\|_{L^2(D)}	\le 
				C\dfrac{ C^2_{\mathrm{inv}}(h)}{ |n|^2}.
			\end{align}
			The constant $C$ depends on $D$, $T$, $\Omega$ and $\kappa$, but is crucially independent of $K$ and the space discretization (whose influence is fully captured in $C_{\mathrm{inv}}(h)$, defined in \eqref{eq:inverse-estimate}).
		\end{theorem}
		\begin{proof}
			Throughout this proof, and for the rest of the paper, $C$ denotes a generic constant with different values that does not depend on $K$ and $h$.
			
			Retracing the substitution used to obtain \eqref{eq:block-EV-h}, we recover the quadratic form of the spatially discrete formulation \eqref{eq:weak-Floquet-freq-quadratic}, which reads
			\begin{align*}
				\begin{aligned}
					&\left(\widehat{v}_h, (-i\omega_{h,K}  + \mathcal D)^2  \widehat{u}_{h,K}^{\omega} \right)_{K,D}
					+ \left(\nabla \widehat{v}_h,  \mathcal T_\kappa \nabla \widehat{u}_{h,K}^{\omega}\right)_{K,D} 	
					\\&
					+\kappa_0 (\widehat{v}_h,  	(-i\omega_{h,K}  + \mathcal D)  \widehat{u}_{h,K}^{\omega} )_{K,\Gamma} 	
					= 0
					,
				\end{aligned} 	
			\end{align*}
			for all $\widehat v_h \in \VhK$.
			We test the formulation with $\widehat{v}_h= (-i\omega_{h,K}  + \mathcal D)^2  \widehat{u}_{h,K}^{\omega} $ and apply the inverse estimate \eqref{eq:inverse-estimate} twice, which gives
			\begin{align*}	
				0 &\ge	\left\|(-i\omega_{h,K}  + \mathcal D)^2  \widehat{u}_{h,K}^{\omega} \right\|^2_{K,D}
				- 	 C^2_{\text{inv}}(h) C_\kappa 	\left\|(-i\omega_{h,K}  + \mathcal D)^2  \widehat{u}_{h,K}^{\omega} \right\|_{K,D}	
				\left\| \widehat{u}_{h,K}^{\omega} \right\|_{K,D}
				\\&
				+\kappa_0 \big((-i\omega_{h,K}  + \mathcal D)^2  \widehat{u}_{h,K}^{\omega},  	(-i\omega_{h,K}  + \mathcal D)  \widehat{u}_{h,K}^{\omega} \big)_{K,\Gamma} 	
				.
			\end{align*}
			Here, we use the constant $C_\kappa$ of Lemma~\ref{lem:bound-a}. We bound the trace pairing from above by the usual trace theorems, which give a constant $C_\Gamma$, such that
			\begin{align*}
				&\big| \big((-i\omega_{h,K}  + \mathcal D)^2  \widehat{u}_{h,K}^{\omega},  	(-i\omega_{h,K}  + \mathcal D)  \widehat{u}_{h,K}^{\omega} \big)_{K,\Gamma} \big|
				\\
				&\le C_\Gamma C_{\text{inv}}(h) \| (-i\omega_{h,K}  + \mathcal D)^2  \widehat{u}_{h,K}^{\omega}\|_{K,D} \|	(-i\omega_{h,K}  + \mathcal D)  \widehat{u}_{h,K}^{\omega} \|_{K,D}.
			\end{align*}
			Inserting this inequality above and rearranging yields
			\begin{align*}	
				&\left\|(-i\omega_{h,K}  + \mathcal D)^2  \widehat{u}_{h,K}^{\omega}\right\|^2_{K,D}\le
				C^2_{\text{inv}}(h) C_\kappa 	\left\|(-i\omega_{h,K}  + \mathcal D)^2  \widehat{u}_{h,K}^{\omega} \right\|_{K,D}	
				\left\| \widehat{u}_{h,K}^{\omega} \right\|_{K,D}
				\\& \quad \quad
				+ C_{\text{inv}}(h) C_\Gamma \| (-i\omega_{h,K}  + \mathcal D)^2  \widehat{u}_{h,K}^{\omega}\|_{K,D} \|	(-i\omega_{h,K}  + \mathcal D)  \widehat{u}_{h,K}^{\omega} \|_{K,D}
				.
			\end{align*}
			Dividing through the common factor and using the normalization of $\widehat{u}_{h,K}^{\omega}$ then yields the intermediate estimate
			\begin{align}\label{eq:identity-localized}
				\begin{aligned}	
					\left\|(-i\omega_{h,K}  + \mathcal D)^2  \widehat{u}_{h,K}^{\omega}\right\|_{K,D}	&\le
					C^2_{\text{inv}}(h) C_\kappa 	
					\left\| \widehat{u}_{h,K}^{\omega} \right\|_{K,D}
					\\&	+ C_{\text{inv}}(h) C_\Gamma \|	(-i\omega_{h,K}  + \mathcal D)  \widehat{u}_{h,K}^{\omega} \|_{K,D}
					\\ & \le 
					C_{\Gamma,\kappa} C_{\text{inv}}(h) \left( C_{\text{inv}}(h)+ |\omega_{h,K} |  +  \|\mathcal D \widehat{u}_{h,K}^{\omega} \|_{K,D}\right)
					.
				\end{aligned}
			\end{align}
			We note that the above term depends quadratically on $|\omega_{h,K} |$ on the left-hand side and linearly on the right-hand side, which yields an estimate from above on $|\omega_{h,K}|$. This computation is performed in the proof of Lemma~\ref{lem:bound-omega-K-h}, which yields the inequality  
			\begin{align*}
				|\omega_{h,K} | \le  C ( C_{\mathrm{inv}}(h) +\|\mathcal D \widehat{u}_{h,K}^{\omega} \|_{K,D}).
			\end{align*}
			Since $\omega_{h,K} $ is in the first Brillouin zone, we obtain a constant $C$, such that for all $|n|>1$, we have, after squaring both sides of \eqref{eq:identity-localized} and applying Young's inequality to the right-hand side
			\begin{align*}
				\|\mathcal D^2 \widehat{u}_{h,K}^{\omega} \|^2_{K,D}	\le 
				C C^2_{\text{inv}}(h) \left( C^2_{\text{inv}}(h) +  \|\mathcal D \widehat{u}_{h,K}^{\omega} \|^2_{K,D} \right).
			\end{align*} 
			Finally, we write out the summands from the norm  $\left\| \cdot\right\|_{K,D}$ and pull all of them to the left-hand side, which gives 
			\begin{align}\label{eq:written-out-K-D-norm}
				\sum_{n=-K}^K \left( |n|^4	-C C^2_{\text{inv}}(h)|n|^2 \right)\left\| \left( \widehat{u}_{h,K}^{\omega}\right)_n \right\|^2_{L^2(D)}	\le 
				C C^4_{\text{inv}}(h).
			\end{align} 
			We conclude with a proof by cases, where we separate the sum at $n_0= \lceil C_0C_{\mathrm{inv}}(h)\rceil$. If the negative terms on the left-hand side of \eqref{eq:written-out-K-D-norm} can be absorbed, i.e., we have
			\begin{align}\label{eq:case-1}
				\tfrac{1}{2}\sum_{|n|>n_0}|n|^4 \left\| \left( \widehat{u}_{h,K}^{\omega}\right)_n \right\|^2_{L^2(D)} 
				>CC^2_{\text{inv}}(h)	\sum_{n=-n_0}^{n_0}  |n|^2 \left\| \left( \widehat{u}_{h,K}^{\omega}\right)_n \right\|^2_{L^2(D)},
			\end{align} 
			then we have the result by absorption and estimating the left-hand side from below by a single summand. 
			
			If \eqref{eq:case-1} does not hold, then we have 
			\begin{align*}
				\sum_{|n|>n_0}|n|^4 \left\| \left( \widehat{u}_{h,K}^{\omega}\right)_n \right\|^2_{L^2(D)} 
				\le 2CC^2_{\text{inv}}(h)	\sum_{n=-n_0}^{n_0}  |n|^2 \left\| \left( \widehat{u}_{h,K}^{\omega}\right)_n \right\|^2_{L^2(D)}
				\le 4CC^4_{\text{inv}}(h)	,
			\end{align*} 
			which already implies the localization result \eqref{eq:local}.
		\end{proof}
For our first main result, we require the following auxiliary result, which gives a bound on the imaginary part of the discrete approximations of the Floquet \JN{exponents} in the first Brillouin zone. 
		\begin{lemma}\label{lem:bound-omega-K-h}
			Let $\omega_{h,K}$ be in the first Brillouin zone and, further, let the assumptions on $\kappa$ formulated in \eqref{assumpt-kappa} hold. Then, there exists a constant $C$ independent of $h$ and $K$, such that 
			\begin{align*}
				|\omega_{h,K}| \le  C ( C_{\mathrm{inv}}(h) + K).
			\end{align*}
		
		\end{lemma}
        
		\begin{numberedproof}{}
			We start from \eqref{eq:identity-localized}, which after rearranging yields 
			\begin{align}\label{eq:bound-w}	
				\begin{aligned}
					|\omega_{h,K}|^2
					& \le |\omega_{h,K} | \|\mathcal D \widehat{u}_{h,K}^{\omega} \|_{K,D} +  \|\mathcal D^2 \widehat{u}_{h,K}^{\omega} \|_{K,D} 
					\\&+
					C_{\Gamma,\kappa} C_{\text{inv}}(h) \left( C_{\text{inv}}(h)+ |\omega_{h,K} |  +   \|\mathcal D \widehat{u}_{h,K}^{\omega} \|_{K,D} \right)
					\\& \le 
					|\omega_{h,K} | (K + C_{\Gamma,\kappa} C_{\text{inv}}(h) ) + K^2 + C_{\Gamma,\kappa} C_{\text{inv}}(h) \left( C_{\text{inv}}(h) +  K \right)	.
				\end{aligned}
			\end{align}
			The statement is now given by applying Young's inequality on the first summand on the right-hand side and taking the square root of both sides. 
			
		\end{numberedproof}
        
        \begin{corollary}
       \JN{In the setting of Lemma~\ref{lem:bound-omega-K-h}and for $K$ sufficiently large, we have
			\begin{align*}
				|\omega_{h,K}| \le  C C^2_{\mathrm{inv}}(h).
			\end{align*}
			Note that $C_{\mathrm{inv}}(h) $ is the constant from the inverse estimate \eqref{eq:inverse-estimate}.}
        \end{corollary}
        \begin{proof}
        		\JN{	We obtain the estimate from the localization result \eqref{eq:local}, which yields a constant $C$, such that
			\begin{align*}
				\|\mathcal D \widehat{u}_{h,K}^{\omega} \|_{K,D} +
				\|\mathcal D^2 \widehat{u}_{h,K}^{\omega} \|_{K,D}
				\le C C^2_{\text{inv}}(h).
			\end{align*}
			Inserting this estimate into \eqref{eq:bound-w} yields the desired bound.}
        \end{proof}

		\subsection{Connection to initial-value problems}
		We continue with a connection between the fully discrete truncated coupled harmonics \eqref{eq:block-EV-h} and the spatially discrete initial value problem \eqref{eq:weak-absorption-h}. 
		%
		
		Let $\omega_{h,K}$ and $\widehat{u}_{h,K}$ denote an eigenvalue-eigenmode pair of the fully discrete system \eqref{eq:block-EV-h}. The associated Floquet--Bloch solution then reads in the time-domain
		\begin{equation*}
			U^{\omega}_{h,K} (t) =  	u^{\omega}_{h,K} (t) e^{-i\omega t}= \sum_{n=-K}^K\left(\widehat{u}^{\omega}_{h,K}\right)_ne^{-i(\omega+n\Omega)t}.
		\end{equation*}
		The following proposition shows that this mode fulfills the spatially discrete time-modulated acoustic wave equation, up to a defect that tends to zero as $K\rightarrow \infty$.
		
		\begin{lemma}\label{lem:big-U}
			Under the assumptions of Theorem~\ref{thm:localization} and the further assumption $\kappa \in C^3_{\mathrm{per}}([0,T];L^{\infty}(D))$, we have the following result. \drop{The temporal approximation $	U^{\omega}_{h,K} (t)$ fulfills the time-modulated weak formulation \eqref{eq:weak-absorption-h} up to a defect, namely} For all $v_h\in X_h$ and $t\ge 0$, it holds that
			\begin{align}\label{eq:weak-Floquet-time-construction}
				\begin{aligned}
					&\left(v_h, \partial_t^2 U^{\omega}_{h,K} (t) \right)_{L^2(D)} + \left(\nabla v_h,  \kappa (t)\nabla U^{\omega}_{h,K}(t) \right)_{L^2(D)} 
					\\&+ \kappa_0\left( v_h,  \partial_t U^{\omega}_{h,K}  (t)\right)_{L^2(\Gamma)} 
					= \left( \nabla v_h, \nabla \delta_{h,K}^\omega(t)\right)_{L^2(D)},
				\end{aligned} 	
			\end{align}
			where the bilinear form of the defect is bounded by 
			\begin{align}\label{eq:bound-delta}
				\sup_{\|v_h\|_{L^2(D)}=1}
				|\left( \nabla v_h, \nabla \delta_{h,K}^\omega(t)\right)_{L^2(D)}|
				\le 
				C	\dfrac{C_{\mathrm{inv}}^4(h)}{K^2}e^{\normalfont{\Im} \omega_{h,K}  t}.
			\end{align}
			Here, $C_{\mathrm{inv}}(h)$ denotes the constant of the inverse estimate \eqref{eq:inverse-estimate} and $C$ depends on $D$, $\kappa$ and $\Omega$.
		\end{lemma}
		\begin{proof}
			
			
			\emph{(i) An associated time-dependent equation.}
			Inserting $U^{\omega}_{h,K}$ and eliminating the terms accounted for in \eqref{eq:weak-Floquet-freq-K} shows that
			$\nabla \delta^\omega_{h,K}$ is the remainder of the Fourier series of $\kappa (t) \nabla U^{\omega}_{h,K}(t)$ from the associated term on the left-hand side. 
			For this term, we have by the construction of the truncated system \eqref{eq:weak-formulation-K}, the explicit form
			\begin{align*}
				\nabla	\delta^\omega_{h,K}(t)= -\sum_{|n|>K}\sum_{m=-K}^K \widehat{\kappa}_{n-m}\left(\nabla \widehat{u}^{\omega}_{h,K}\right)_m e^{-i(\omega_{h,K}+n\Omega)t}.
			\end{align*}
			\emph{(ii) Bounds on the remainder.}
			We note that the bilinear form is continuous with respect to the appropriate norm of Proposition~\ref{prop:energy-estimate}, since the inverse estimate \eqref{eq:inverse-estimate} implies that
			\begin{align}
				|\left( \nabla v_h, \nabla \delta_{h,K}^\omega(t)\right)_{L^2(D)}|
				\le 
				C^2_{\mathrm{inv}}(h) \| v_h\|_{L^2(D)}
				\| \delta_{h,K}\|_{L^2(D)}.
			\end{align}
			\JN{Using the triangle inequality therefore gives
			\begin{align}\label{eq:delta-est}
				 \| \delta_{h,K} (t)
				\|_{L^2(D)}\,
				&\le
				e^{\Im \omega_{h,K}  t}
				\sum_{|n|>K} \sum_{m=-K}^K \left\|	\widehat{\kappa}_{n-m}\right\|_{L^\infty(D)}\left\|
				\left(\widehat{u}^{\omega}_{h,K}\right)_m\right\|_{L^2(D)}.
			\end{align}}
			With $\kappa\in C^3_{\text{per}}(0,T;L^{\infty}(D))$, the localization estimate \eqref{eq:local} then yields a constant $C$, such that 
			\begin{align*}
				& \| \delta_{h,K} (t)
				\|_{L^2(D)}\,
				\le C e^{\Im \omega_{h,K}  t} C_{\mathrm{inv}}^2(h)
				\sum_{m=-K}^K \dfrac{1}{1+|m|^2}	\sum_{|n|>K}  \dfrac{1}{|n-m|^3}
				\\&\quad  \le
				C	C_{\mathrm{inv}}^2(h)e^{\Im \omega_{h,K}  t} 
				\sum_{m=-K}^K \dfrac{1}{(1+|m|^2)(K+1-m)^2}
				\le 	C	\dfrac{C_{\mathrm{inv}}^2(h)}{K^2}e^{\Im \omega_{h,K}  t}.
			\end{align*}
		\end{proof}
		\begin{remark}\label{rem:delta-better}
			The factor $e^{\normalfont{\Im} \omega_{h,K}  t}$ appears in multiple estimates, due to its appearance in \eqref{eq:bound-delta}. At this point, this factor can be interpreted as problematic, since  $|\omega_{h,K}|$ is only bounded by Lemma~\ref{lem:bound-omega-K-h}, which includes an inverse estimate. Combining Lemma~\ref{lem:big-U} with the energy estimate in Proposition~\eqref{prop:energy-estimate-h}, however, yields the bounds of Theorem~\ref{thm:truncated_Floquet-h} which explicitly estimates the imaginary part of $\omega_{h,K}$. Inserting these into the right-hand side of \eqref{eq:bound-delta}, yields 
			\begin{align}\label{eq:bound-delta-better}
				\sup_{\|v_h\|_{L^2(D)}=1}
				|\left( \nabla v_h, \nabla \delta_{h,K}^\omega(t)\right)_{L^2(D)}|
				\le 
				C	\dfrac{C_{\mathrm{inv}}^4(h)}{K^2}e^{C_{\kappa}'  t},
			\end{align}
			under the assumptions of Theorem~\ref{thm:truncated_Floquet-h} (for sufficiently large $K$).
		\end{remark}
		The computational construction of Floquet--Bloch modes via the system \eqref{eq:block-EV-h} can be understood as approximations of special initial values, whose solution of the initial value problem \eqref{eq:weak-absorption-h} without excitation then almost fulfill the form of the Floquet--Bloch mode \eqref{eq:Bloch}. 
		Motivated by this, we define an alternative approximation of the same Floquet--Bloch mode, we define $u^{h,K,\omega}_{\mathrm{init}}$ as that solution to the initial value problem \eqref{eq:weak-absorption-h} that initially agrees with $	U^{\omega}_{h,K}$, namely 
		\begin{equation}\label{eq:init-u-init}
			u^{h,K,\omega}_{\mathrm{init}}(0)=U^{\omega}_{h,K} (0)
			\quad\text{and}\quad 
			\partial_t u_{\mathrm{init}}^{h,K,\omega} (0)= \partial_t U^{\omega}_{h,K}(0).
		\end{equation}
		By classical arguments based on semigroup theory, this is a well-posed definition (see \cite{K73}). The following result gives some basic properties of those functions. Initially, both functions agree, but since $	U^{\omega}_{h,K}$ is not an exact solution of the time-modulated acoustic wave equation, a cumulative difference gradually arises between them. The following result explicitly estimates their difference.	
		\begin{proposition}\label{prop:td-sol-differences}
			Under the assumptions of Theorem~\ref{thm:localization}, we consider $u^{h,K,\omega}_{\mathrm{init}}(t)$ defined by the initial value problem with the initial conditions \eqref{eq:init-u-init} and $U^{\omega}_{h,K}$ as it is characterized in Lemma~\ref{lem:big-U}.
			Let $d(t)=	u^{h,K,\omega}_{\mathrm{init}}(t)- U^{\omega}_{h,K}(t)$ denote the difference of both approximations.
			We then have the estimate
			\begin{align*}
				&	\left\| \partial_t d(t)
				\right\|_{L^2(D)}+
				\left\| \nabla	d(t)
				\right\|_{L^2(D)}
				\le 	C   	\dfrac{C_{\mathrm{inv}}^4(h)}{K^2} .
			\end{align*}
			The function $u^{h,K,\omega}_{\mathrm{init}}(t)$ therefore fulfills the quasi-periodic temporal boundary conditions up to a defect, namely
			\begin{align*}
				\|\normalfont{ \Re} u^{h,K,\omega}_{\mathrm{init}}(0)e^{-i\omega_{h,K}T}-\normalfont{ \Re} u^{h,K,\omega}_{\mathrm{init}}(T)\|_{L^2(D)}\le 	C   	\dfrac{C_{\mathrm{inv}}^4(h)}{K^2}.
			\end{align*}
			The constant $C$ is crucially independent of $h$ and $K$, but depends on $t$,$T$,$\Omega$, $\kappa$ and $D$.
		\end{proposition}
		\begin{proof}
			The difference $d(t)$ fulfills, by construction,  the initial value problem in Lemma~\ref{lem:big-U} with the vanishing initial conditions. Applying the energy estimate \eqref{prop:energy-estimate-h} then yields the result using \eqref{eq:bound-delta-better} (obtained with Theorem~\ref{thm:truncated_Floquet-h}, as discussed in Remark~\ref{rem:delta-better}) to bound the remainder term.
		\end{proof}
		\subsection{Restrictions on $\omega_{h,K}$}
		The results of Lemmas~\ref{lem:FL-multipliers-absorption}--\ref{lem:FL-multipliers-neumann} extend to the fully discrete approximations $\omega_{h,K}$, under the assumptions of Theorem~\ref{thm:localization} that ensure the localization property \eqref{eq:local}. 
		
		\begin{theorem}\label{thm:truncated_Floquet-h}
			Let the setting of Theorem~\ref{thm:localization} hold and 
			further assume that $\kappa \in C^3_{\mathrm{per}}([0,T];L^{\infty(D)})$.
			Consider the eigenvalues of the system \eqref{eq:block-EV-h} in the first Brillouin zone with the absorbing boundary condition \eqref{eq:Absorbing}. 
			Then, there exists a constant $C_0$ independent of $h$ and $K$ such that for
			\begin{align}\label{eq:cond-loc}
				K > C_0C_{\mathrm{inv}}^2(h),
			\end{align}
			we have the following result.
			The fully discrete resonant quasi-frequencies lie on a complex half-space, namely
			\begin{align*}
				\normalfont\Im \omega_{h,K} \le {C_{\kappa}'}.
			\end{align*}
			Moreover, in the case of homogeneous Neumann boundary condition \eqref{eq:Neumann}, we have that the discrete approximations $\omega_{h,K}$ are in strip near the real line, namely
			\begin{align*}
				|	\normalfont\Im \omega_{h,K}| \le {C_{\kappa}'}.
			\end{align*}   
			We note that the constant ${C_{\kappa}'}$ is defined in Proposition~\ref{prop:energy-estimate}.
		\end{theorem}
		\begin{proof}
			
			%
			The energy of the Floquet--Bloch solution (more precisely, the energy of $ U^{\omega}_{h,K} (t) $) fulfills, due to the quasi-periodicity, the identity
			\begin{align}\label{eq:energy-Bloch}
				\mathcal E_h(T) = 	e^{ \Im\omega_{h,K}  2T}\mathcal E_h(0).
			\end{align}
			
			Without loss of generality, we assume $\mathcal E_h(0) = 1$.
			From the energy estimate with excitation in Proposition~\ref{prop:energy-estimate-h} (which is slightly stronger at a full period $t=T$) and estimate \eqref{eq:bound-delta} on the remainder, we then obtain 
			\begin{align}\label{eq:before-square root}
				\begin{aligned}
					e^{ 2\Im\omega_{h,K}  T}
					&\le 
					e^{C_{\kappa}' T} \bigg(1+ \int_0^{T}   \dfrac{C^4_{\mathrm{inv}}(h)}{4C_{\kappa}'}\| \delta^\omega_{h,K} (s)\|^2_{L^2(D)} \mathrm d s\bigg)
					\\ & \le e^{C_{\kappa}'T} 
					\bigg(1+C
					\dfrac{ C^8_{\mathrm{inv}}(h)}{ K^4}
					e^{2\Im \omega_{h,K}  T}
					\bigg) .
				\end{aligned}
			\end{align}
			We now derive the restriction \eqref{eq:cond-loc} on $K$, where the constant $C_0$ is determined such that it ensures that
			\begin{align*}
				e^{-C_{\kappa}'T}
				\le  1-				\dfrac{ C C^8_{\mathrm{inv}}(h)e^{C_{\kappa}'T}}{ K^4} .
			\end{align*}
			Rearranging \eqref{eq:before-square root} and using this condition then yields
			\begin{align*}
				e^{-C_{\kappa}'T}	e^{ 2\Im\omega_{h,K}  T}\le e^{C_{\kappa}'T} .
			\end{align*}

			%
			%
			The statement for the Neumann problem is deduced in the same way and completed by using the same argument for the time-reversed solution.
		\end{proof}		
		\begin{remark}[On the restriction of $K$]\label{rem:restrict-K}
			We note that the assumption on $K$ in Theorem~\ref{thm:truncated_Floquet-h} is generally quite severe. For localized modes, this assumption can be relaxed. Let $\omega_{h,K}$ be in the first Brillouin zone and associated to a more regular eigenmode $\widehat{u}^\omega_{h,K}$. Further, let the spectrum of this mode be more localized than the estimate in Theorem~\ref{thm:localization} guarantees, namely for some constant $C_{\mathrm{loc}}$ and $p\ge 1$,
			\begin{align*}
				\left\|\left(\widehat{u}^\omega_{h,K} \right)_n \right\|_{L^2(D)}\le \dfrac{C_{\mathrm{loc}}}{1+|n|^p}
				\quad \text{and}\quad 
				\kappa \in C^{p+1}_{\mathrm{per}}(0,T).
			\end{align*} 
			By following the proof of Theorem~\ref{thm:truncated_Floquet-h}, we obtain a constant $C_0$ such that under the condition
			\begin{align*}
				K^{p} > C_0C_{\mathrm{loc}} C_{\mathrm{inv}}^2(h),
			\end{align*}
			the result of Theorem~\ref{thm:truncated_Floquet-h} holds. The original condition \eqref{eq:cond-loc} is then the special case of $C_{\mathrm{loc}}=C_{\mathrm{inv}}^2(h)$ and $p=2$, which is ensured by Theorem~\ref{thm:localization}, although numerical evidence (see Figure~\ref{Fig:localization}) suggests that many modes fulfill stronger conditions.			
		\end{remark}

		\subsection*{A discrete folding property}
		From the explicit form of the Floquet--Bloch solution \eqref{eq:Bloch}, we note that $\omega\in\mathbb C$ and $u^\omega\in V_{\mathrm{per}}$ identify the same mode via
		\begin{equation}\label{eq:cont-folding}
			\omega = \omega+l\Omega \in \mathbb C\quad \text{and}\quad u^\omega_l = e^{il\Omega t}u^\omega\in V_{\mathrm{per}},
		\end{equation} 
		for arbitrary $l\in  \mathbb Z$, leading to the same expression in \eqref{eq:Bloch}. We rewrite this property in terms of the Fourier coefficients of a quasi-resonant mode.

        		We denote by $\mathrm{F}^l : \VK\rightarrow \VK$ the operator determined by \eqref{eq:cont-folding} in the Fourier domain, determined by
		\begin{equation}\label{Folding-op}\mathrm{F}^l \widehat v = (\widehat v_{n-l})_{n=-K}^K,
		\end{equation} where the entries of $\widehat v$ are periodically extended via $$\widehat v_{n+j(2K+1)} = \widehat v_n,
        \quad \text{ for all }j\in \mathbb Z.$$ 
		For a given kernel element $(\omega,u^\omega)$  of the left-hand side of the system with infinitely many harmonics \eqref{eq:weak-Floquet-freq} (i.e., setting $K=\infty$), we have that for all $ l \in \mathbb Z$ it holds that
		\begin{align}\label{eq:folding}
			\sup_{\|\widehat v \|_{\VK} = 1} \left| a_{\omega+l\Omega}\left(\mathrm{F}^l\widehat{u}^\omega \, ,\, \widehat{v}\right) \right| =0.
		\end{align}
		
		Such a property cannot generally be expected from the truncated system \eqref{eq:weak-Floquet-freq-K}. A weaker result can be shown: For modes that are localized, we find that the residuum of the folded quasi-frequencies tends to zero, as $K\rightarrow \infty$ for a fixed $l$ . The following \JN{theorem} gives this result.
		\begin{theorem}\label{thm:discrete-folding}
			Consider the setting of Theorem~\ref{thm:localization} and let the time-modulation be slightly more regular than assumed in \eqref{assumpt-kappa}, namely assume $\kappa \in C^3_{\mathrm{per}}([0,T];L^{\infty}(D))$. 
			Furthermore, let $l\le K$ be constant and $\mathrm{F}^l$ denote the associated folding operator defined in \eqref{Folding-op}. 
			Then, there exists a constant $C_0$ independent of $h$ and $K$, such that for
			\begin{align*}
				K > C_0C_{\mathrm{inv}}^2(h),
			\end{align*}
			we have that
			\begin{align}\label{eq:discrete-folding}
				\sup_{\|\widehat v_h \|_{\VhK} = 1} \left| a_{\omega^l_{h,K}}^K\left(\mathrm{F}^l\widehat{u}^\omega_{h,K} \, ,\, \widehat{v}_h\right) \right| \le 	C	\dfrac{C_{\mathrm{inv}}^4(h)}{K^2}.
			\end{align}
			Here we used the notation $\omega^l_{h,K} = \omega_{h,K}+l\Omega$ for the folded discrete resonant quasi-frequency.
			The constant $C$ on the right-hand side is independent of $K$ and the space discretization, but depends on $D$, $T$, $\Omega$ and $\kappa$. The constant $C_{\mathrm{inv}}(h)$ is defined in \eqref{eq:inverse-estimate}.
		\end{theorem}
		\begin{proof}
			The terms appearing here are a subset of the remainder estimated in Proposition~\ref{prop:td-sol-differences}, which gives the result by using Theorem~\ref{thm:truncated_Floquet-h} as described in Remark~\ref{rem:delta-better}. 
		\end{proof}
        
		\section{Implementation and numerical experiments}
		As a toy problem, we consider the simple one-dimensional setting, with the reduced form of \eqref{eq:time-varying-acoustic}, which reads
		\begin{align*}
			\partial_t^2  u(x,t) - \kappa(t)\partial_{xx} u(x,t)  = 0, \quad x\in (-1,1), \,t \ge 0 .
		\end{align*}
		On the boundary $\{-1,1\}$, we either impose the absorbing boundary condition \eqref{eq:Absorbing} with $\kappa_0=\tfrac{1}{2}$ or the homogeneous boundary condition \eqref{eq:Neumann}.
		As the time modulation, we choose a function with infinitely many active Fourier nodes, namely
		\begin{align*}
			\kappa(t) = 1+ \tfrac{1}{10}e^{\cos(t)},
		\end{align*}
		which has the period $\Omega=2\pi$. We seek Floquet exponents of this system, which are approximated by the eigenvalues of the fully discrete block system \eqref{eq:block-EV-h}.
		
		Generally, in the dimensions $d=2,3$, the classical Galerkin subspaces $V_h$ of interest are sets of piecewise polynomials on a mesh, discretizing the domain $D$. Here, we validate our results with a one-dimensional setting. 
		
		As the finite-dimensional subspace $V_h$ , we use a spectral Galerkin ansatz, namely we set
		$
		V_h = \mathcal P_p ([-1,1]),
		$
		which denotes the set of polynomials of maximal degree $p\in \mathbb N$. More details on spectral Galerkin methods and their implementation can be found in \cite{T00}. In the following, we outline the setup of our experiments, whose results are presented in Figures~\ref{Fig:Discrete-spectrum}--\ref{Fig:localization}.
		
		\begin{figure}[h]
			\centering
			\hspace*{-1.2cm}
			\includegraphics[scale=0.6]{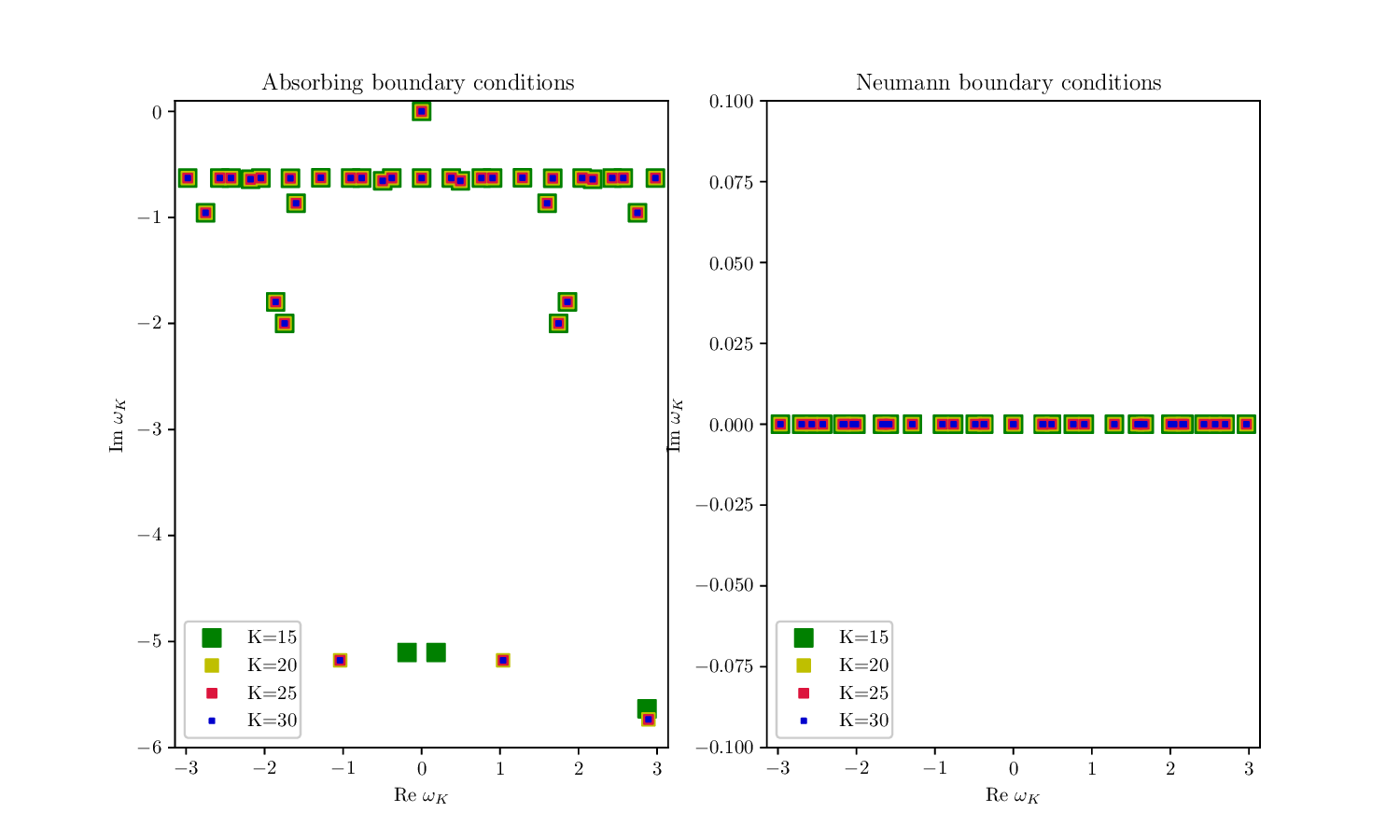}
			\caption{A visualization of those discrete eigenvalues of the system \eqref{eq:block-EV-h}, whose real part lies in the first Brillouin zone (i.e., $\text{Re}\, \omega_{h,K}  \in [-\pi,\pi]$. The plot was generated with a spectral Galerkin discretization of degree $p=15$.}
			\label{Fig:Discrete-spectrum}
		\end{figure}
		\emph{A view on the fully discrete spectrum:} Our first results are shown in Figure~\ref{Fig:Discrete-spectrum}, which shows all discrete eigenvalues whose real part are within the first Brillouin zone $(-\Omega/2,\Omega/2]=(-\pi,\pi]$. Here, we used a fixed Galerkin subspace of degree $p=10$ and varying truncation $K$. On the left, we consider the case of the absorbing boundary condition and on the right plot, we imposed homogeneous Neumann boundary conditions. For sufficiently large $K$, the approximations show a reasonable agreement. Both systems have a Floquet exponent in the origin (which corresponds to the constant function $u(x,t) = \chi_D(x)$) and the discrete Floquet exponents lie in the predicted regions of Theorem~\ref{thm:truncated_Floquet-h}. 
		\begin{figure}
			\centering
			\hspace*{-1.2cm}
			\includegraphics[scale=0.5]{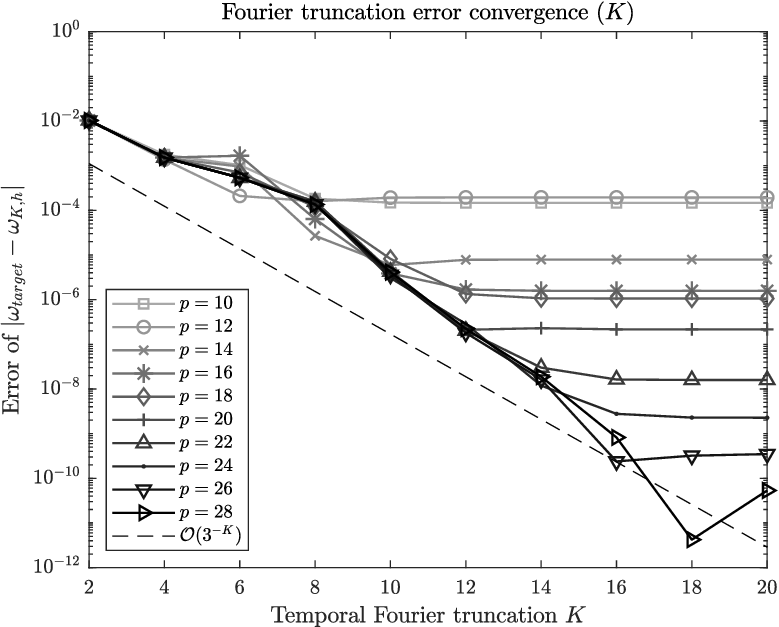}
			\caption{Convergence plot of a target eigenvalue $\omega\approx-0.976i$ with respect to the Fourier parameter $K$, with varying spectral Galerkin space discretizations.}
			\label{Fig:conv-plot-K}
		\end{figure}
		
		\begin{figure}
			\centering
			\hspace*{-1.2cm}
			\includegraphics[scale=0.5]{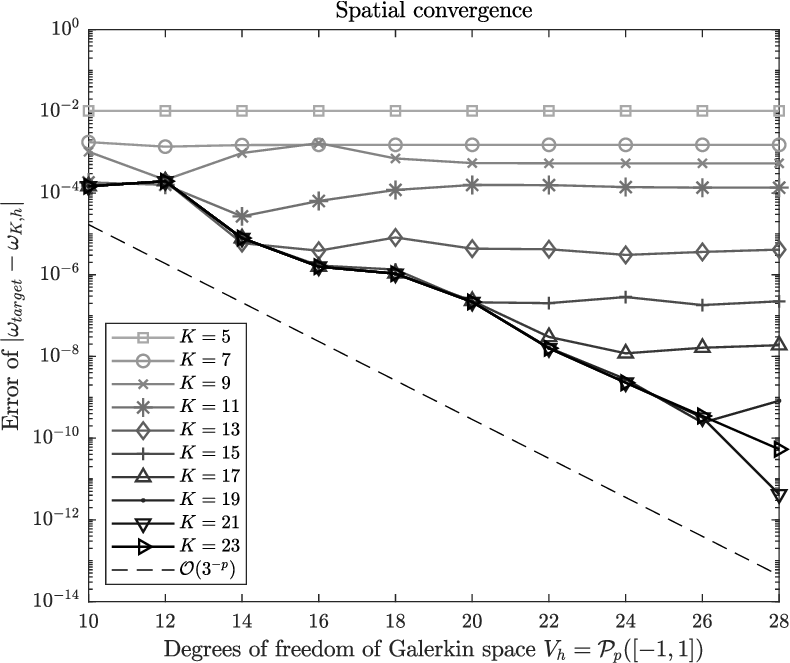}
			\caption{Convergence plot of a target eigenvalue $\omega\approx-0.976i$ with respect to the degrees of freedom of the spectral Galerkin space discretization, for varying  Fourier truncations.}
			\label{Fig:conv-plot-p}
		\end{figure}
		
		\emph{Convergence plots with respect to $K$ and $p$:} We continue with convergence plots for the fully discrete system with absorbing boundary conditions, which are found in Figures~\ref{Fig:conv-plot-K}--\ref{Fig:conv-plot-p}. We compare a target eigenvalue of different discretizations of the system \eqref{eq:block-EV-h} with an eigenvalue computed with a reference solution, where we chose $K=30$ and $p=40$. As the target eigenvalue, we choose the value directly below the origin in the left plot of Figure~\ref{Fig:Discrete-spectrum}, which is roughly located around ${\omega\approx-i}$. Figure~\ref{Fig:conv-plot-K} now fixes different space discretizations (different $p$) and varies the Fourier truncation parameter $K=2,4,6,\dots,20$. The approximation error decays exponentially, roughly with the rate $3^{-K}$, until the space discretization error dominates. Conversely, Figure~\ref{Fig:conv-plot-p} fixes different truncation parameters $K$ and varies the polynomial degree of the spectral Galerkin method. As $p$ increases, we again observe an exponential decay of the approximation error, roughly of the same order. \JN{In view of the slowing decay of the Fourier coefficients observed in Figure~\ref{Fig:eigenmodes} for $p\rightarrow \infty$, we suspect that the error behavior corresponds to a particularly well-behaved mode, which fulfills a stronger decay estimate than the worst-case estimate of Theorem~\ref{thm:localization}. In general, this can not be expected and with our current results, we cannot explain this behavior.}

        \JN{
		Figures~\ref{Fig:conv-plot-K}--\ref{Fig:conv-plot-p} indicate that the approximate Floquet exponents converge when both $p,K\rightarrow \infty$. This behavior can, so far, not be explained theoretically due to several significant obstacles:
        \begin{itemize}
        \item The set of Floquet exponents is generally a continuous spectrum (see Appendix~\ref{sect:time-independent}). Single Floquet exponents can therefore not be isolated, which complicates the convergence analysis.
        \item The weak formulation of the eigenvalue problem \eqref{eq:weak-time-Floquet} does not seem to be associated to an operator that has the classical decomposition "coercive + compact". This is the usual route to establish regular convergence of the discrete operators (see, e.g., \cite{H21}) and would enable convergence results (see \cite{K96a}). 
        \item The defects in Lemma~\ref{lem:big-U}, Proposition~\ref{prop:td-sol-differences} and Theorem~\ref{thm:discrete-folding} converge, under the stated conditions of the relation of $p$ and $K$, to zero. However, the overall rate of convergence is only polynomial of low degree, much slower than the observed exponential rate (with respect to $K$, when setting $K\propto p^3$). Here, some form of regularity of the approximated mode $\widehat{u}^\omega$ could be used to derive a stronger form of the localization estimate of Theorem~\ref{thm:localization}, along the lines of Remark~\ref{rem:restrict-K}. 
        \end{itemize}
 		}
		\begin{figure}
			\centering
			\includegraphics[scale=0.5]{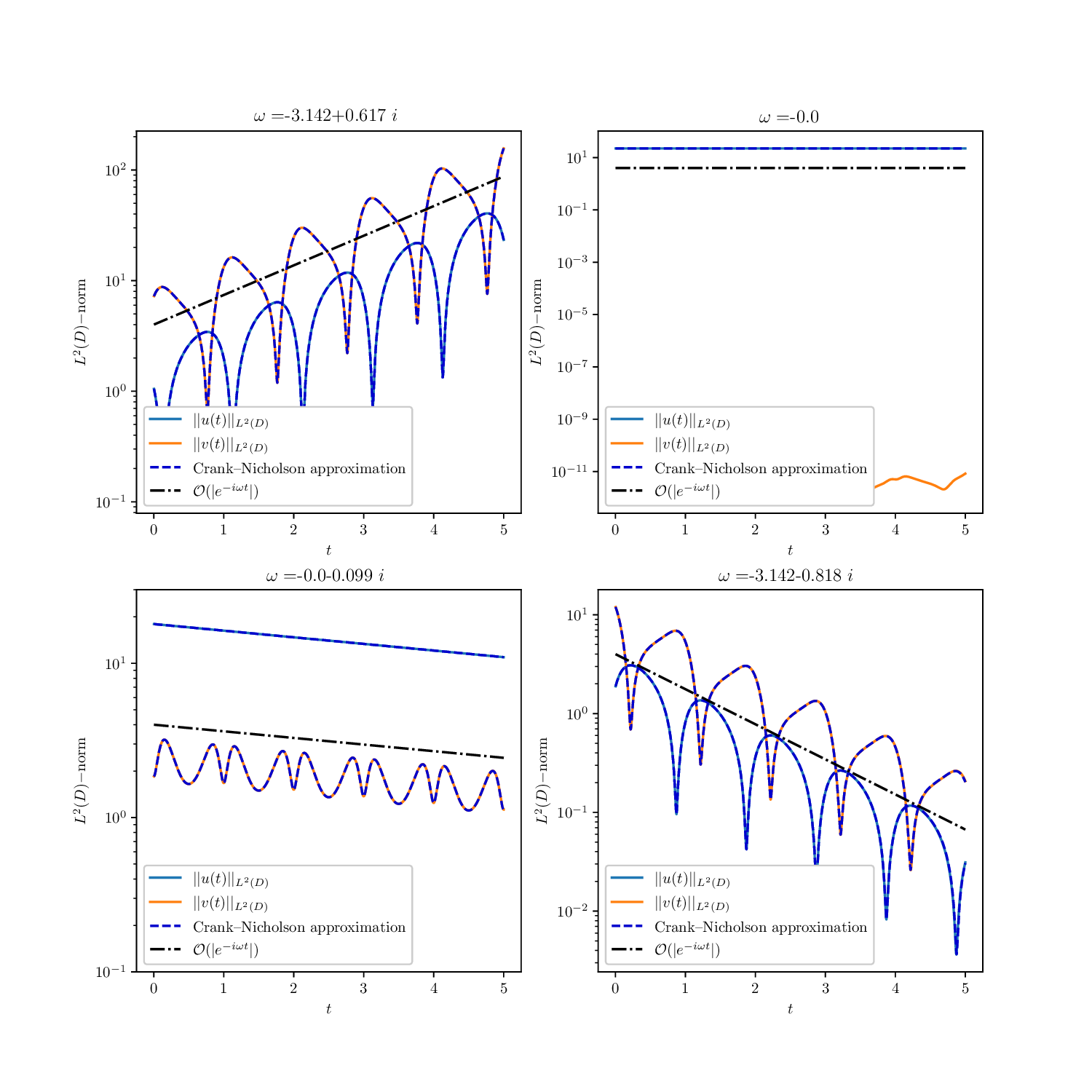}
			\vspace*{-1cm}
			\caption{Time-dependent norms of the solutions corresponding to different eigenmodes of the system \eqref{eq:block-EV-h} with $K=25$ and $p=20$, as well as a numerical approximations based on a Crank--Nicholson scheme with $N=2000$ timesteps, applied to \eqref{eq:weak-absorption-h}.}
			\label{Fig:time_Floquets}
		\end{figure}
		\begin{figure}
			\centering
			\hspace*{-1cm}
			\includegraphics[scale=0.5]{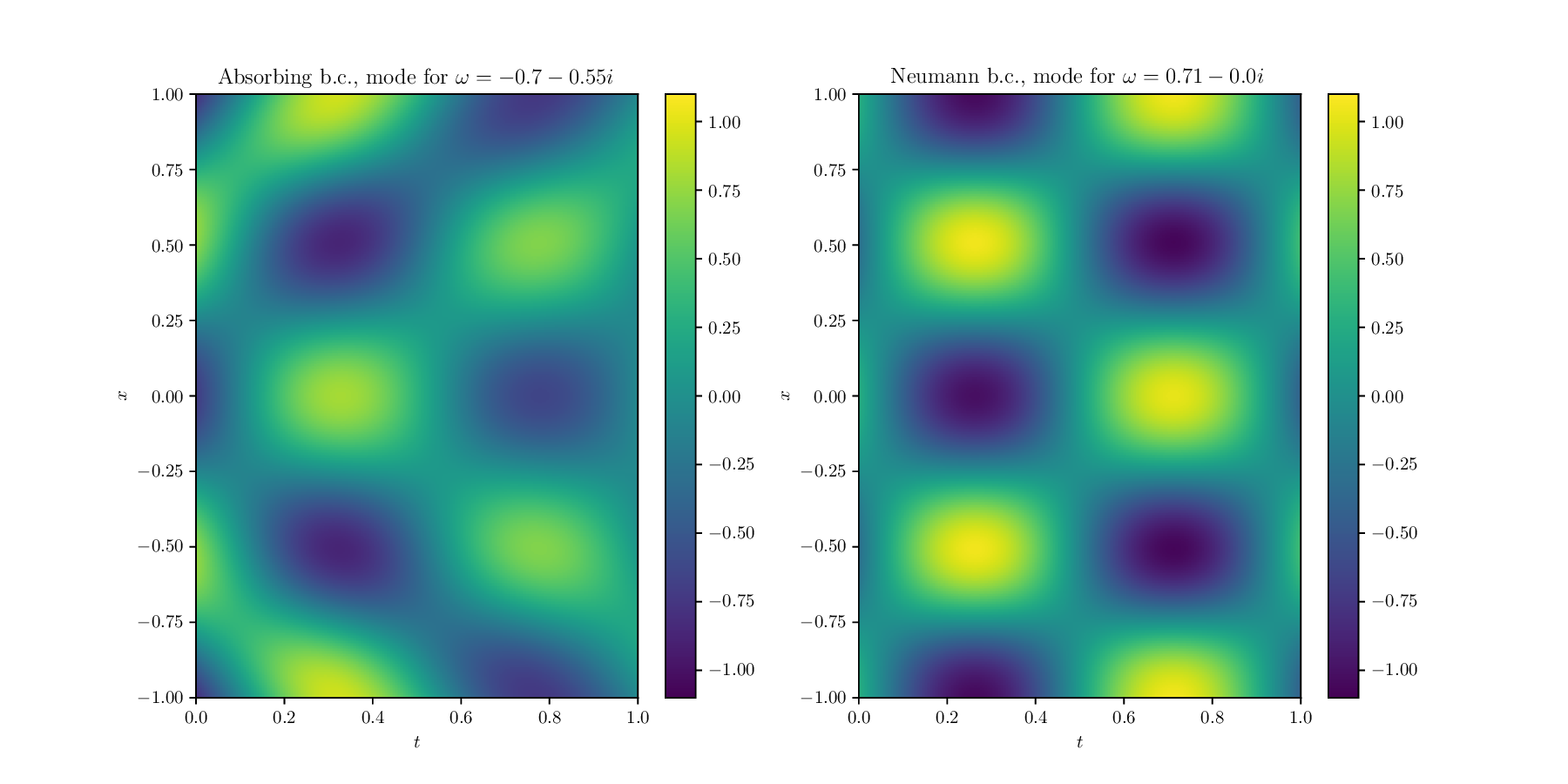}
			\caption{Two Floquet solutions, of the absorbing boundary condition \eqref{eq:Absorbing} on the left plot and the Neumann boundary condition \eqref{eq:Neumann} on the right on a single period $[0,T]$. 
            }
			\label{Fig:eigenmodes}
		\end{figure}
		\emph{Validation in the time domain:} We validate our approximations of Floquet solutions $u^\omega$ for a few values of $\omega$, by approximating the solutions of the initial value problem, whose initial values are determined by the evaluation of Floquet solutions corresponding to different $\omega$ at $t=0$. More precisely, we use a Crank--Nicholson time integration scheme to approximate the function $u^{h,K,\omega}_{\mathrm{init}}$, whose difference to the temporal mode associated with the approximated Fourier coefficients $\widehat{u}^\omega_{h,K}$ was analyzed in Proposition~\ref{prop:td-sol-differences}. In order to observe exponentially growing modes, we use $\kappa(t) = e^{\cos(2\pi t)}$ and crucially use the smaller parameter $\kappa_0=0.1$, which reduces the absorption from the absorbing boundary conditions. The solutions are then approximated over $5$ periods (until the final time $\widetilde T = 5T=5$) and the norms of the respective solutions are shown in Figure~\ref{Fig:time_Floquets}. We observe that the predicted decay of the Floquet exponent and the mode align well with the solution generated by the time integration scheme. As predicted by Proposition~\ref{prop:td-sol-differences} (for sufficiently large $K$), their difference is small. This strong agreement is observed despite the neglection of the restrictive assumption on $K$ (in Proposition~\ref{prop:td-sol-differences}), since the plots were created with $K=25$ and $p=20$. This may be caused by a stronger localization property (than the worst-case estimate \eqref{eq:local}) of the visualized modes.
		
		Figure~\ref{Fig:eigenmodes} visualizes two Floquet solutions, which are inhibited in the system with absorbing boundary conditions \eqref{eq:Absorbing} or Neumann boundary conditions \eqref{eq:Neumann}. The visualized approximations have been computed with $K=25$ and $p=10$.
		
		\emph{On frequency localization of eigenmodes:} Truncating the system \eqref{eq:weak-Floquet-freq-K} with a finite $K$ introduces a defect whose magnitude depends on the decay of the Fourier coefficients of both $\kappa$ and $u_\omega$. 
       While the decay of $\widehat \kappa_n$ is determined by the regularity of the given physical constants $\kappa$, a decay of the Fourier coefficients of the (unknown) $u_\omega$ would be desirable but can generally not be assumed. Figure~\ref{Fig:localization} shows the norms of the Fourier coefficients of the eigenmodes with a fixed truncation $K=30$ and a sequence of increasingly refined space discretizations. For a coarse space discretization, which prohibits high spatial oscillations, the spectrum is strongly localized around $0$ and an exponential decay of the temporal Fourier coefficients is observed. This could be explained by the connection of spatial and temporal oscillations in Lemma~\ref{lem:energy-identity-with-excit}, although a complete understanding of this phenomenon remains elusive. As spatial discretization becomes more expressive, we observe an increasing amount of modes with stronger temporal oscillations, which corresponds to a weaker localization of the temporal spectrum near $0$. For those modes, we do not expect that the truncation of the system yields a satisfactory result, since the truncation \eqref{eq:weak-Floquet-freq-K} then introduces a defect of $\mathcal O(1)$. This implies a spectral CFL-type condition: For any given space discretization, we require a sufficiently large $K$, to resolve all present modes (whose existence and completeness is guaranteed by Floquet-theory, formulated in Lemma~\ref{lem:Floquet}) to reduce the truncation error of \eqref{eq:weak-Floquet-freq-K} to a satisfactory tolerance. On the other hand, we observe that low-frequency modes are present even for space discretizations that violate such a condition. Here, we expect a good agreement of the associated modes from the truncated system \eqref{eq:weak-Floquet-freq-K} with their continuous counterparts in \eqref{eq:weak-Floquet-freq} (under the assumption of existence). 
		\begin{figure}
			\centering
			\includegraphics[scale=0.5]{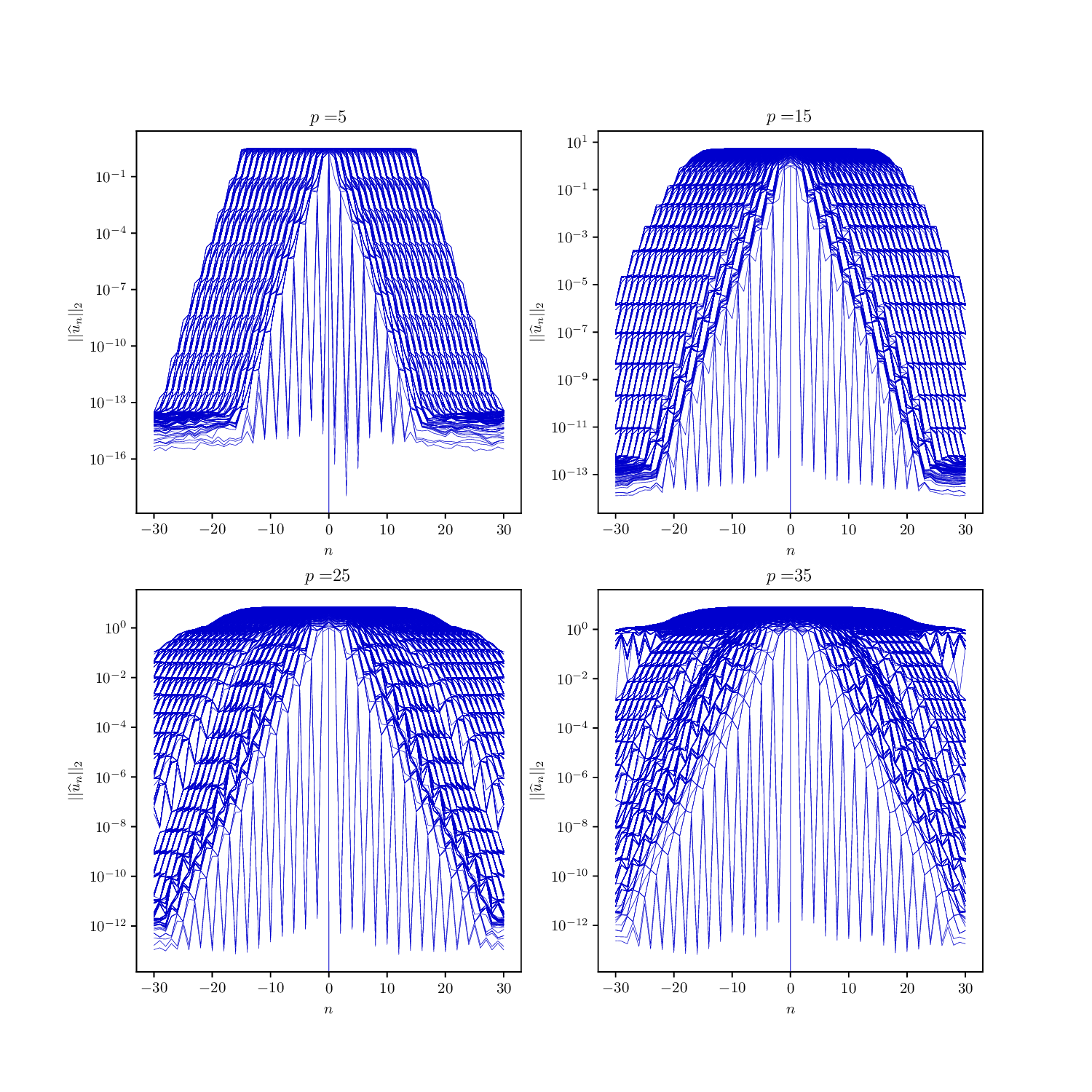}
			\vspace*{-1cm}
			\caption{The $L^2-$ norm of the Fourier coefficients of the truncated Floquet--Bloch solutions in \eqref{eq:weak-Floquet-freq-K}, computed for $\kappa(t)=\cos(\cos(2\pi t))$. For coarse space discretizations, the modes localize in the low-frequency regime. 
            }
			\label{Fig:localization}
		\end{figure}


		\appendix
		
		\section{Time-periodic acoustic wave propagation problems with time independent coefficients}\label{sect:time-independent}
        The spectrum visualized in  Figure~\ref{Fig:Discrete-spectrum} seems to fill up a complex contour (with a resolution that is determined by the space discretization). The following section gives a simple example where the set of Floquet \JN{exponents} densely fills out a complex contour. 
		\subsection{Explicit folding in one dimension}\label{sect:1d-folding}

		Consider the simplest setting of $D=[0,1]$, with homogeneous Neumann boundary conditions \eqref{eq:Neumann}. The solution is then in the span of the associated complete basis $\phi_k(x)=\cos(k\pi x)$. Then, the solution to the time-modulated acoustic wave equation has the form $\sum_{k=-\infty}^\infty\alpha_k(t)\phi_k(x)$, where the temporal coefficients fulfill, for $\kappa(x,t)=\kappa(t)$,
		\begin{align}\label{eq:1D-Fourier-coefficients}
			(-i\omega+\partial_t)^2\alpha_k(t)+\kappa(t)\pi^2k^2\alpha_k(t) = 0,
			\quad \forall |k|\le K.
		\end{align}
		Again, for the decoupled case $\kappa(t)=\kappa$, inserting a Fourier approach yields
		\begin{align*}
			\left((-i\omega-in\Omega)^2+\kappa\pi^2k^2\right)\widehat{\alpha}_k^n = 0,
			\quad \forall |k|\le K.
		\end{align*}
		Injectivity, for real $\omega$, is therefore equivalent to 
		\begin{align*}
			(\omega+n\Omega)\neq \sqrt{\kappa}\pi k, 
			\quad \forall k,n.
		\end{align*}
		For $\kappa=\Omega=1$, we observe that this condition is equivalent to
		\begin{align*}
			\omega\neq \pi k -n,
			\quad \forall k,n.
		\end{align*}
		The resonant quasi-frequencies are therefore, in most cases, a dense subset of $\mathbb T$. 

		Now consider the same setting with absorbing boundary conditions \eqref{eq:Absorbing}, again with $\kappa=\Omega=1$, instead of the previous Neumann boundary conditions \eqref{eq:Neumann}. Consider the problem for an interval, namely we consider the problem 
		\begin{align*}
			-k^2 \widehat u - \Delta \widehat u = 0 \quad \text{on} \quad [-1,1],
		\end{align*}
		with perfectly absorbing boundary conditions (which correspond to a coupling to an exterior Helmholtz problem with the material parameter $\kappa_0\in (0,1)$)
		\begin{align*}
			\dfrac{\mathrm d}{\mathrm d |x| } \widehat u = -ik \kappa_0\widehat u \quad \text{on} \quad \partial [-1,1] = \{-1,1\}.
		\end{align*}
		A weak formulation of this boundary value problem is given by
		\begin{align*}
			-k^2 \left(v,\widehat u\right)_{L^2}
			+ \left(\nabla v, \nabla \widehat u\right)_{L^2}
			+ ik\kappa_0 \left( v(-1)\widehat u(-1) -
			v(1)u(1)\right) = 0,
		\end{align*}
		We are interested in nontrivial solutions of this system. Any solution of the interior Helmholtz problem has the form 
		\begin{align*}
			\widehat u = c_1 e^{i k x} + c_2 e^{-ikx}.
		\end{align*}
		Taking the derivative of  this expression yields
		\begin{align*}
			\dfrac{\mathrm d}{\mathrm d x }    \widehat u
			= c_1 i k e^{ikx} + c_2(-ik) e^{-ikx}.
		\end{align*}
		The absorbing boundary conditions at both borders of the interval $[-1,1]$ therefore correspond to the matrix system 
		\begin{align*}
			\kappa\begin{pmatrix}
				e^{ik} & -e^{-ik} \\ 
				-e^{-ik} & e^{ik}
			\end{pmatrix}
			c = \kappa_0 \kappa
			\begin{pmatrix}
				e^{ik} & e^{-ik} \\ 
				e^{-ik} & e^{ik}
			\end{pmatrix} c,
		\end{align*}
		where $c\in \mathbb C^2$ collects the coefficients of the solution $\widehat u $. 
        
        \drop{We observe that for $\kappa=0$ (which corresponds to a constant resonant mode) we have a nontrivial kernel. For $\kappa\neq 0$ we obtain by rearranging this system
		\begin{align*}
			Ac \coloneqq \begin{pmatrix}
				(1-\kappa_0)e^{ik} & -(1+\kappa_0)e^{-ik} \\ 
				-(1+\kappa_0)e^{-ik} & (1-\kappa_0)e^{ik}
			\end{pmatrix}
			c = 0.
		\end{align*}
		We seek frequencies $k_j$, such that the above matrix has a nontrivial kernel. The determinant of the matrix reads 
		\begin{align*}
			\det(A) = (1-\kappa_0)^2 e^{ik} - (1+\kappa_0)^2 e^{-ik}.
		\end{align*}
		Setting this determinant to zero, rearranging and taking the square root on both sides gives, for a resonant frequency $k_{\text{res}}$, the identity
		\begin{align*}
			e^{ik^{\text{res}}} = \dfrac{1+\kappa_0}{1-\kappa_0}.
		\end{align*}
		We therefore obtain that the time-harmonic resonance frequencies are, for $\kappa_0\in(0,1)$, fully described by 
		\begin{align*}
			\widehat{\mathcal R} = \{0\} \cup \left\{ 2\pi k - i \log\left(\dfrac{1+\kappa_0}{1-\kappa_0}\right) \quad \bigg| \quad k \in \mathbb Z\right\}.
		\end{align*}
        }
		A standard computation now gives that the quasi-resonant frequencies in the time-periodic setting are given by the set
		\begin{align*}
			\mathcal R = \mathbb Z \,\cup \bigg\{ 2\pi k+ n - i \log\left(\dfrac{1+\kappa_0}{1-\kappa_0}\right) \,\, \bigg| \,\, k,n \in \mathbb Z\bigg\},
		\end{align*}
		which contains a dense subset of the complex line $\mathbb L = \big\{z \in \mathbb C \, | \, \Im z = -\log\big(\tfrac{1+\kappa_0}{1-\kappa_0}\big)\big\}$. We note that in the special case of $\kappa_0 = 1$, no quasi-resonant frequencies are present.

        \section{Improved estimates for Neumann b.c. with $H^1$ normalization}
        The localization result of Theorem~\ref{thm:localization} can, in the case of the Neumann problem, be slightly improved by using a stronger normalization. As a consequence, the corresponding results of Lemma~\ref{lem:big-U}, Proposition~\ref{prop:td-sol-differences} and of Theorem~\ref{thm:discrete-folding} hold, under this stronger normalization, with a lower power of the constant $C_{\mathrm{inv}}(h)$. The following theorem formulates and proves this statement.
        
        \begin{theorem}\label{thm:stronger-localization-neumann} Consider the problem formulation for the homogeneous Neumann boundary conditions, i.e. $\kappa_0=0$. Let $\kappa$ fulfill the assumptions \eqref{assumpt-kappa} and further let $\widehat{u}_{h,K}^{\omega} $ and $\omega_{h,K} $ be a resonant mode and its associated quasi-frequency, which fulfill the eigenvalue problem \eqref{eq:block-EV-h}. Let $\omega_{h,K} $ be in the first Brillouin zone, i.e. have real part in $(-\Omega/2,\Omega/2]$. Moreover, we assume that the mode is normalized in the sense that 
        $$\left\| \widehat{u}_{h,K}^{\omega}\right\|_{K,D}^2+\left\| \nabla\widehat{u}_{h,K}^{\omega}\right\|_{K,D}^2 = 1.$$ 
        Then, there exists a constant $C_0$ independent of $h$ and $K$ such that for
			\begin{align}\label{eq:cond-loc-ap}
				K > C_0C_{\mathrm{inv}}^2(h),
			\end{align}
			we have the following result.
            For $1<|n|\le K$, we obtain
			\begin{align}\label{eq:local-ap}
				\left\| \left( \widehat{u}_{h,K}^{\omega}\right)_n \right\|_{L^2(D)}	\le 
				C\dfrac{ C_{\mathrm{inv}}(h)}{ |n|^2}.
			\end{align}
			The constant $C$ depends on $D$, $T$, $\Omega$ and $\kappa$, but is crucially independent of $K$ and the space discretization (whose influence is fully captured in $C_{\mathrm{inv}}(h)$, defined in \eqref{eq:inverse-estimate}).
		\end{theorem}
		\begin{proof}
			Throughout this proof, and for the rest of the paper, $C$ denotes a generic constant with different values that does not depend on $K$ and $h$.
			
			Retracing the substitution used to obtain \eqref{eq:block-EV-h}, we recover the quadratic form of the spatially discrete formulation \eqref{eq:weak-Floquet-freq-quadratic}, which reads
			\begin{align*}
				\begin{aligned}
					&\left(\widehat{v}_h, (-i\omega_{h,K}  + \mathcal D)^2  \widehat{u}_{h,K}^{\omega} \right)_{K,D}
					+ \left(\nabla \widehat{v}_h,  \mathcal T_\kappa \nabla \widehat{u}_{h,K}^{\omega}\right)_{K,D} 	
					= 0
					,
				\end{aligned} 	
			\end{align*}
			for all $\widehat v_h \in \VhK$.
			We test the formulation with $\widehat{v}_h= (-i\omega_{h,K}  + \mathcal D)^2  \widehat{u}_{h,K}^{\omega} $ and apply the inverse estimate \eqref{eq:inverse-estimate} once, which gives
			\begin{align*}	
				0 &\ge	\left\|(-i\omega_{h,K}  + \mathcal D)^2  \widehat{u}_{h,K}^{\omega} \right\|^2_{K,D}
				- 	 C_{\text{inv}}(h) C_\kappa 	\left\|(-i\omega_{h,K}  + \mathcal D)^2  \widehat{u}_{h,K}^{\omega} \right\|_{K,D}	
				\left\| \nabla \widehat{u}_{h,K}^{\omega} \right\|_{K,D}
				.
			\end{align*}
			Here, we use the constant $C_\kappa$ of Lemma~\ref{lem:bound-a}. 
			Rearranging this inequality yields
			\begin{align*}	
				&\left\|(-i\omega_{h,K}  + \mathcal D)^2  \widehat{u}_{h,K}^{\omega}\right\|_{K,D}
                \le
				C_{\text{inv}}(h) C_\kappa 
				\left\|\nabla  \widehat{u}_{h,K}^{\omega} \right\|_{K,D}
				.
			\end{align*}
			
			We note that by Theorem~\ref{thm:truncated_Floquet-h}, we have $|\omega_{h,K} |\le C_{\kappa}'$. Since $\omega_{h,K} $ is in the first Brillouin zone, we obtain a constant $C$, such that for all $|n|>1$, we have, after squaring both sides of the above inequality and applying Young's inequality to the right-hand side
			\begin{align*}
				\|\mathcal D^2 \widehat{u}_{h,K}^{\omega} \|^2_{K,D}	\le 
				C \left( C^2_{\text{inv}}(h) +  \|\mathcal D \widehat{u}_{h,K}^{\omega} \|^2_{K,D} \right).
			\end{align*} 
			Finally, we write out the summands from the norm  $\left\| \cdot\right\|_{K,D}$ and pull all of them to the left-hand side, which gives 
			\begin{align}\label{eq:written-out-K-D-norm}
				\sum_{n=-K}^K \left( |n|^4	-C |n|^2 \right)\left\| \left( \widehat{u}_{h,K}^{\omega}\right)_n \right\|^2_{L^2(D)}	\le 
				C C^2_{\text{inv}}(h).
			\end{align} 
			We conclude with a proof by cases, where we separate the sum at $n_0= \lceil C_0\rceil$. If the negative terms on the left-hand side of \eqref{eq:written-out-K-D-norm} can be absorbed, i.e., we have
			\begin{align}\label{eq:case-1}
				\tfrac{1}{2}\sum_{|n|>n_0}|n|^4 \left\| \left( \widehat{u}_{h,K}^{\omega}\right)_n \right\|^2_{L^2(D)} 
				>C\sum_{n=-n_0}^{n_0}  |n|^2 \left\| \left( \widehat{u}_{h,K}^{\omega}\right)_n \right\|^2_{L^2(D)},
			\end{align} 
			then we have the result by absorption and estimating the left-hand side from below by a single summand. 
			
			If \eqref{eq:case-1} does not hold, then we have 
			\begin{align*}
				\sum_{|n|>n_0}|n|^4 \left\| \left( \widehat{u}_{h,K}^{\omega}\right)_n \right\|^2_{L^2(D)} 
				\le 2C\sum_{n=-n_0}^{n_0}  |n|^2 \left\| \left( \widehat{u}_{h,K}^{\omega}\right)_n \right\|^2_{L^2(D)}
				\le 4CC^2_{\text{inv}}(h)	,
			\end{align*} 
			which again implies the localization result \eqref{eq:local}.
		\end{proof}

		\bibliographystyle{abbrv}
		\bibliography{Lit}

\begin{thebibliography}{10}

\bibitem{ACH22}
H.~Ammari, J.~Cao, and E.~O. Hiltunen.
\newblock Nonreciprocal wave propagation in space-time modulated media.
\newblock {\em Multiscale Model. Simul.}, 20(4):1228--1250, 2022.

\bibitem{ACHR23}
H.~Ammari, J.~Cao, E.~O. Hiltunen, and L.~Rueff.
\newblock Transmission properties of time-dependent one-dimensional
  metamaterials.
\newblock {\em J. Math. Phys.}, 64(12):Paper No. 121502, 18, 2023.

\bibitem{ACHR24}
H.~Ammari, J.~Cao, E.~O. Hiltunen, and L.~Rueff.
\newblock Scattering from time-modulated subwavelength resonators.
\newblock {\em Proc. A.}, 480(2289):Paper No. 20240177, 22, 2024.

\bibitem{AZ18}
H.~Ammari, B.~Fitzpatrick, H.~Kang, M.~Ruiz, S.~Yu, and H.~Zhang.
\newblock {\em Mathematical and computational methods in photonics and
  phononics}, volume 235 of {\em Mathematical Surveys and Monographs}.
\newblock American Mathematical Society, Providence, RI, 2018.

\bibitem{AH21}
H.~Ammari and E.~O. Hiltunen.
\newblock Time-dependent high-contrast subwavelength resonators.
\newblock {\em J. Comput. Phys.}, 445:Paper No. 110594, 18, 2021.

\bibitem{Thea1}
H.~Ammari, E.~O. Hiltunen, and T.~Kosche.
\newblock Asymptotic {F}loquet theory for first order {ODE}s with finite
  {F}ourier series perturbation and its applications to {F}loquet
  metamaterials.
\newblock {\em J. Differential Equations}, 319:227--287, 2022.

\bibitem{Thea2}
H.~Ammari and T.~Kosche.
\newblock Topological phenomena in honeycomb {F}loquet metamaterials.
\newblock {\em Math. Ann.}, 388(3):2755--2785, 2024.

\bibitem{B43}
R.~Bellman.
\newblock The stability of solutions of linear differential equations.
\newblock {\em Duke Math. J.}, 10:643--647, 1943.

\bibitem{BS08}
S.~C. Brenner and L.~R. Scott.
\newblock {\em The mathematical theory of finite element methods}, volume~15 of
  {\em Texts in Applied Mathematics}.
\newblock Springer, New York, third edition, 2008.

\bibitem{CD20}
C.~Caloz and Z.-L. Deck-Léger.
\newblock Spacetime metamaterials—part i: General concepts.
\newblock {\em IEEE Transactions on Antennas and Propagation},
  68(3):1569--1582, 2020.

\bibitem{C19}
Y.~Chen, X.~Li, H.~Nassar, A.~N. Norris, C.~Daraio, and G.~Huang.
\newblock Nonreciprocal wave propagation in a continuum-based metamaterial with
  space-time modulated resonators.
\newblock {\em Phys. Rev. Appl.}, 11:064052, Jun 2019.

\bibitem{G19}
E.~Galiffi, P.~A. Huidobro, and J.~B. Pendry.
\newblock Broadband nonreciprocal amplification in luminal metamaterials.
\newblock {\em Phys. Rev. Lett.}, 123:206101, Nov 2019.

\bibitem{G22}
E.~Galiffi, R.~Tirole, S.~Yin, H.~Li, S.~Vezzoli, P.~A. Huidobro, M.~G.
  Silveirinha, R.~Sapienza, A.~Al{\`u}, and J.~B. Pendry.
\newblock Photonics of time-varying media.
\newblock {\em Advanced Photonics}, 4(1):014002--014002, 2022.

\bibitem{H21}
M.~Halla.
\newblock Galerkin approximation of holomorphic eigenvalue problems: weak
  {T}-coercivity and {T}-compatibility.
\newblock {\em Numer. Math.}, 148(2):387--407, 2021.

\bibitem{HSW23}
S.~N. Hameedi, A.~Sagiv, and M.~I. Weinstein.
\newblock Radiative decay of edge states in {F}loquet media.
\newblock {\em Multiscale Model. Simul.}, 21(3):925--963, 2023.

\bibitem{HD24}
E.~O. Hiltunen and B.~Davies.
\newblock Coupled harmonics due to time-modulated point scatterers.
\newblock {\em Phys. Rev. B}, 110:184102, Nov 2024.

\bibitem{K96a}
O.~Karma.
\newblock Approximation in eigenvalue problems for holomorphic {F}redholm
  operator functions. {I}.
\newblock {\em Numer. Funct. Anal. Optim.}, 17(3-4):365--387, 1996.

\bibitem{K73}
T.~Kato.
\newblock Linear evolution equations of ``hyperbolic'' type. {II}.
\newblock {\em J. Math. Soc. Japan}, 25:648--666, 1973.

\bibitem{K16}
P.~Kuchment.
\newblock An overview of periodic elliptic operators.
\newblock {\em Bull. Amer. Math. Soc. (N.S.)}, 53(3):343--414, 2016.

\bibitem{RMA03}
J.~J. Rico, M.~Madrigal, and E.~Acha.
\newblock Dynamic harmonic evolution using the extended harmonic domain.
\newblock {\em IEEE Transactions on Power Delivery}, 18(2):587--594, 2003.

\bibitem{SS11}
S.~A. Sauter and C.~Schwab.
\newblock {\em Boundary element methods}, volume~39 of {\em Springer Series in
  Computational Mathematics}.
\newblock Springer-Verlag, Berlin, 2011.
\newblock Translated and expanded from the 2004 German original.

\bibitem{S78}
W.~Szemplińska-Stupnicka.
\newblock The generalized harmonic balance method for determining the
  combination resonance in the parametric dynamic systems.
\newblock {\em Journal of Sound and Vibration}, 58(3):347--361, 1978.

\bibitem{T00}
L.~N. Trefethen.
\newblock {\em Spectral methods in {MATLAB}}, volume~10 of {\em Software,
  Environments, and Tools}.
\newblock Society for Industrial and Applied Mathematics (SIAM), Philadelphia,
  PA, 2000.

\bibitem{WB19}
X.~Wang and F.~Blaabjerg.
\newblock Harmonic stability in power electronic-based power systems: Concept,
  modeling, and analysis.
\newblock {\em IEEE Transactions on Smart Grid}, 10(3):2858--2870, 2019.

\bibitem{W90}
N.~M. Wereley.
\newblock {\em Analysis and control of linear periodically time varying
  systems}.
\newblock PhD thesis, Massachusetts Institute of Technology, 1990.

\bibitem{WH90}
N.~M. Wereley and S.~R. Hall.
\newblock Frequency response of linear time periodic systems.
\newblock In {\em 29th IEEE conference on decision and control}, pages
  3650--3655. IEEE, 1990.

\bibitem{WG20}
Z.~Wu and A.~Grbic.
\newblock Serrodyne frequency translation using time-modulated metasurfaces.
\newblock {\em IEEE Transactions on Antennas and Propagation},
  68(3):1599--1606, 2020.

\bibitem{YGA22}
S.~Yin, E.~Galiffi, and A.~Al{\`u}.
\newblock Floquet metamaterials.
\newblock {\em ELight}, 2(1):8, 2022.

\bibitem{YXB19}
X.~Yue, X.~Wang, and F.~Blaabjerg.
\newblock Review of small-signal modeling methods including frequency-coupling
  dynamics of power converters.
\newblock {\em IEEE Transactions on Power Electronics}, 34(4):3313--3328, 2019.

\end{thebibliography}
		
	\end{document}